\documentclass[a4paper,twoside,english,reqno,moreauthors]{amsart} 
\usepackage{latexsym,amsfonts,amssymb,verbatim,mathrsfs,cite,verbatim,amsmath,amsthm}
\usepackage[mathscr]{euscript}
\usepackage{mathptmx}
\usepackage{babel}
\usepackage{enumerate}
\usepackage{comment}
\usepackage{soul,cite,lineno,cancel}
\usepackage{booktabs,multirow,slashbox,verbatim,relsize,tikz,tikz-cd,rotating,lscape}
\usepackage{enumitem}
\usepackage{mathtools}
\usepackage{academicons}
\usepackage{tikz,xcolor,hyperref}

\definecolor{lime}{HTML}{A6CE39}
\DeclareRobustCommand{\orcidicon}{
	\begin{tikzpicture}
	\draw[lime, fill=lime] (0,0) 
	circle [radius=0.16] 
	node[white] {{\fontfamily{qag}\selectfont \tiny ID}};
	\draw[white, fill=white] (-0.0625,0.095) 
	circle [radius=0.007];
	\end{tikzpicture}
	\hspace{-2mm}
}
\foreach \x in {A, ..., Z}{%
	\expandafter\xdef\csname orcid\x\endcsname{\noexpand\href{https://orcid.org/\csname orcidauthor\x\endcsname}{\noexpand\orcidicon}}
}

\newtheorem{theorem}{Theorem}[section]
\newtheorem{corollary}[theorem]{Corollary} 
\newtheorem{axiom}[theorem]{Axioms} 
 
\newtheorem{proposition}[theorem]{Proposition}
\newtheorem{assumption}[theorem]{Assumption}
\newtheorem{assumptions}[theorem]{Assumptions}

\newtheorem{definition}[theorem]{Definition}

\theoremstyle{definition}
\newtheorem{remark}[theorem]{Remark}

\newtheorem{example}[theorem]{Example}

\title[Abstract integration ...]{Abstract integration with respect to 
measures and applications to modular convergence %in Orlicz spaces
in vector lattice setting}

\author[A. Boccuto]{Antonio Boccuto, \,\,
 \orcidA{} }
\author[A. R. Sambucini]{Anna Rita Sambucini, \,\,
 \orcidB{} }
\address{Department of Mathematics and Computer Sciences, University of Perugia, 1, Via Vavitelli, 06123 Perugia - Italy}
\email{antonio.boccuto@unipg.it; anna.sambucini@unipg.it}
\begin{document}
\begin{abstract}
A "Bochner-type" integral for  vector lattice-valued functions with respect to (possibly infinite)
vector lattice-valued measures is presented with respect to abstract 
convergences, satisfying suitable axioms, and some fundamental properties are studied. Moreover, by means of this integral,
some convergence results on operators in vector lattice-valued modulars are
proved. Some applications are given to moment kernels and to the Brownian motion.
\end{abstract}   
\maketitle 
%================================
%================================
\section{Introduction}\label{sec1}
In the literature there are many studies concerning
the problem of approximating a real-valued function $f$ by Urysohn-type integral operators  
or discrete sampling operators.
These topics  together with some other kind of operators,  have several applications in several branches,
for instance neural networks and  reconstruction of signals and images 
(see, e.g., \cite{AV2,A2, A3,bm1, bmmellin, BCG2016, BMV, CCCGV, CCGV2021,
CG2017, CV2016,CSV,butz1,butz2}).

In the classical setting, a signal is viewed as a function $f$, defined on a suitable finite or infinite
time interval, which is reconstructed by starting of  
sampled values of the type $f\Bigl(\dfrac{k}{w}\Bigr)$, where $k$ 
is a suitable (finite or not) subset of $\mathbb{Z}$, 
and $w$ is the theoretical minimum sampling rate  necessary to reconstruct  the signal entirely
(see, e.g., \cite{BMV} and the references therein).
Thus, the original unknown signal can be considered as a  function with values in a space consisting of 
random variables, that is (classes of equivalence of) measurable functions. Among them, 
the spaces $L^0$ of $L^2$ are widely used  in stochastic processes. In these kind of spaces, 
it is often advisable  to deal with almost everywhere convergence, which 
does not have a topological nature, but corresponds 
to order convergence, if we endow the involved space with the 
"component-wise" order and supremum. Thus, it is natural to 
investigate \emph{vector lattice-valued functions}, 
defined on a set of (possibly infinite) measure, like for example 
the Lebesgue measure on a halfline or the whole of $\mathbb{R}$
(see, e.g., \cite{BCSVITALI}). 

In this paper we extend to the vector lattice setting the problem of approximating a function $f$ by means of
Urysohn-type integral operators in  the setting of modular convergence, extending some earlier results
proved in \cite{bm1,bmmellin,BCSVITALI,BDMEDITERRANEAN}. We consider three vector lattices 
${\mathbf{X}}_1$, ${\mathbf{X}}_2$, ${\mathbf{X}}$, 
"linked" by a suitable "product" structure, we deal with  ${\mathbf{X}}_1$-valued functions defined
on a metric space, (possibly infinite) ${\mathbf{X}}_2$-valued measures and construct an
${\mathbf{X}}$-valued integral, which will be an extension of  the integrals investigated in \cite{BC2009}, where only finite
measures are considered, and in \cite{BCSVITALI}, where 
it is supposed that ${\mathbf{X}}$ and ${\mathbf{X}}_1$ are
endowed with stronger order units $e$ and $e_1$, respectively. We
drop this hypothesis in our context. We  endow ${\mathbf{X}}_1$, ${\mathbf{X}}_2$
and ${\mathbf{X}}$ with abstract convergences and 
structures of "limit superior", satisfying suitable axioms, 
in order to include usual relative uniform and order convergence  and the related almost and Ces\`{a}ro 
convergences in the general case, the order filter convergence
at least in the $L^p$ case (with $0 \leq p \leq \infty$), and the 
relative uniform filter convergence when ${\mathbf{X}}$ and ${\mathbf{X}}_1$ have order units. Note that this type 
of convergence is equivalent to the norm filter convergence 
in ${\mathbb{R}}$, where we have the norm generated  by the order unit.

The paper is structured as follows. 
In Section \ref{due} we present the abstract axioms on convergence  and limit superior and give some examples. 
In Section \ref{product} we develop integration theory in our setting, giving  Vitali-type theorems and a version of the Lebesgue 
dominated convergence theorem,  we compare our integral with the classical Lebesgue integral, and state
some main propertiers, among which  some versions of Jensen's inequality, in connection with uniformly continuous 
and convex functions. The proof of these results will be given 
in the Appendix. In Section \ref{modularriesz} we recall  the theory of vector lattice-valued modulars introduced in
\cite{BCSVITALI}.
In Section \ref{structural} we give the structural hypotheses on the
involved operators, and in Section \ref{argument} we present our main 
results on modular convergence of operators in Orlicz spaces 
in the vector lattice setting.
In Subsection \ref{appl}, as examples and applications, we deal with 
moment kernels in the vector lattice context, and with
It\^{o}-type integrals with respect to Brownian motion.
%================================
%================================
%================================
\section{Preliminaries}\label{due}
Let  $(\mathbf{X}, \leq_X)$ be a vector lattice and  let the symbols $\vee$ and $\wedge$ denote the 
{\em lattice suprema} and {\em infima}
 in $\mathbf{X}$. We say that 
 $\mathbf{X}$ is  \textit{Dedekind complete} iff every nonempty subset $A \subset \mathbf{X}$, 
order bounded from above, admits a  lattice supremum  in $\mathbf{X}$, denoted by $\bigvee A$.
From now on, $\mathbf{X}$ is a Dedekind complete vector lattice,   $\mathbf{X}^{+}:=\{ x \in \mathbf{X}: x \geq_X 0\}$ 
 is
its positive cone,  and we denote by 
the symbol $\mathbb{R}_0^+$ 
(resp., $\mathbb{R}^+$), as usual, the 
set of non-negative (resp., strictly positive) real numbers. 
For  every $x \in \mathbf{X}$, set $\vert   x\vert    = x \vee (-x)$.\\
We add to $\mathbf{X}$ an extra 
 element 
$+\infty$, extending the order and the operations  on $\mathbf{X}$ in a natural way. Let
$\overline{\mathbf{X}}=\mathbf{X} \cup \{+ \infty\}$,   $\overline{\mathbf{X}}^{+}=\mathbf{X}^{+} \cup \{+ \infty\}$, 
and assume, by convention, $0 \cdot (+ \infty) = 0$.\\
An \emph{$(o)$-sequence} $(\sigma_l)_{l \in \mathbb{N}}$ in  $\mathbf{X}^{+}$ is a decreasing sequence with
$\wedge_l \sigma_l=0$ (see, e.g., \cite{BRV}). For what unexplained we refer, for example,  to \cite{BC2009,BDBOOK}.\\
A \emph{strong order unit} of a vector lattice ${\mathbf{X}}$ is an element
$e\in {\mathbf{X}}^+ \setminus \{0\}$ with the property that, for every $x \in {\mathbf{X}}$,
there exists $\delta \in {\mathbb{R}}^+$ such that $\vert   x\vert    \leq_X \delta \, e$.
For  example, if  $e\in {\mathbf{X}}^+ \setminus \{0\}$,
 then $e$ is a strong order  unit of $V[e]:=\{x \in {\mathbf{X}} : \text{    there   is  } \delta \in 
\mathbb{R}^+$   with    $\vert   x\vert    \leq_X \delta \, e \}.$
 Moreover we observe that $V[e]$ is \emph{solid} in ${\mathbf{X}}$ and 
Dedekind complete.
Furthermore,  by virtue of the Kakutani-Krein theorem (see, e.g.,  \cite[Theorem II.7.4]{SCHAEFER}), $V[e]$ is algebraically and lattice 
isomorphic to the space  of all real-valued continuous functions 
defined on a suitable compact 
and extremely disconnected topological space $\Omega$
(see, e.g., \cite{KAWABE2}).
%\end{example}
A similar, more general theorem 
(Maeda-Ogasawara-Vulikh-type representation theorem)
holds also for any  Dedekind complete vector lattice  
$\mathbf{X} \hookrightarrow   \mathcal{C}_{\infty}(\Omega) =
\{f \in \widetilde{\mathbb{R}}^{\Omega} : f 
\mbox{ is continuous, and   } 
\{\omega: \vert  f(\omega) \vert =+\infty \}  \mbox{ is nowhere dense in }
\Omega \}$
(see, e.g., \cite[Theorem 2.1]{FILTER}, \cite{MO, VULIKH}.
%=================================
%=================================
%================================
%=================================
%================================
We will use  the symbols $\sup$ and $\inf$ to denote both the pointwise suprema and infima in ${\mathcal C}_{\infty}(\Omega)$ and 
those in the (extended) real line. \\

Let ${\mathcal T}$ be the set of all sequences $(x_n)_n$ in $\mathbf{X}$ and ${\mathcal T}^{+}=\{ (x_n)_n \in {\mathcal T}$:
$x_n \geq_X 0$ for each $n \in \mathbb{N}\}$.
Now we give an axiomatic approach to convergence in vector lattices (see, e.g., \cite[Definition 2.1]{BC2009} and \cite{bbdmkorovkin}). 
\begin{axiom} \label{convergenze}
\rm 
A \em convergence \rm is a pair $(\mathcal{S},\ell)$, where $\mathcal{S}$ is a linear subspace of ${\mathcal T}$ and $\ell$ is a
function $\ell:\mathcal{S} \to \mathbf{X}$,  which satisfies the following  conditions: 
\begin{itemize} \label{convergenceaxioms}
\item[{\ref{convergenze}.a)}] 
$\ell((\zeta_1 \, x_n + \zeta_2 \, y_n)_n)= \zeta_1 \, \ell ( (x_n)_n)+ \zeta_2 \, \ell ( (y_n)_n)$
for every pair of sequences $(x_n)_n$, $(y_n)_n \in \mathcal{S}$ and for each $\zeta_1$, $\zeta_2 \in \mathbb{R}$ (\em linearity\rm ).

\item[{\ref{convergenze}.b)}] 
If $(x_n)_n$, $(y_n)_n \in \mathcal{S}$ and $x_n \leq_X y_n$ definitely, then ${\ell}((x_n)_n) \leq_X \ell((y_n)_n)$ (\em monotonicity\rm ).

\item[{\ref{convergenze}.c)}] 
If $(x_n)_n$ is such that $x_n=l$ definitely, then $(x_n)_n \in \mathcal{S}$ and $\ell((x_n)_n)=l$; if $(x_n)_n$, $(y_n)_n$
are such that the set $\{n \in \mathbb{N}: x_n \neq y_n \}$ is finite and $(x_n)_n \in \mathcal{S}$, then $(y_n)_n \in \mathcal{S}$ 
and $\ell((y_n)_n)=\ell((x_n)_n)$.

\item[{\ref{convergenze}.d)}] 
If $(x_n)_n \in \mathcal{S}$, then $(\vert   x_n\vert   )_n \in \mathcal{S}$ and $\ell((\vert   x_n\vert   )_n)=\vert   \ell((x_n)_n)\vert   $.

\item[{\ref{convergenze}.e)}]
Given three sequences $(x_n)_n$, $(y_n)_n$, $(z_n)_n$, 
 with $(x_n)_n$, $(z_n)_n \in \mathcal{S}$, 
$\ell((x_n)_n)=\ell((z_n)_n)$, and  $x_n \leq_X y_n \leq_X z_n$ definitely, then $(y_n)_n \in \mathcal{S}$ 
(and hence from \ref{convergenze}.b) it follows that $\ell((x_n)_n)= \ell((y_n)_n)=\ell((z_n)_n))$.

\item[{\ref{convergenze}.f)}]
 If $u \in \mathbf{X}^{+}$, then the sequence $\Bigl(\dfrac1n u\Bigr)_n$ belongs to $\mathcal{S}$ and 
$\ell\Bigl(\Bigl(\dfrac1n u\Bigr)_n\Bigr)=0$.
\end{itemize}
\end{axiom}

The next property  is a consequence of \ref{convergenze}.c) and \ref{convergenze}.e).
\begin{itemize}
\item[{\ref{convergenze}.g)}] 
If $x \in \mathbf{X}$, $(y_n)_n \in \mathcal{S}$ and $x \leq_X y_n$ definitely, then  $x \leq_X \ell((y_n)_n)$.
\end{itemize}

Analogously, we now present an axiomatic approach of  a ``limit superior''-type vector lattice-valued
operator related  to a convergence $(\mathcal{S}, \ell)$, 
 satisfying Axioms \ref{convergenceaxioms} 
(see, e.g., \cite[Section 2]{bbdmkorovkin}).

\begin{axiom}\label{limsuppresentation}
\rm Let ${\mathcal T}$, $\mathcal{S}$ be as in Axioms \ref{convergenze} and define a function $\overline{\ell}: {\mathcal T}^{+}
\to \overline{\mathbf{X}}^{+}$  satisfying the following 
 conditions:
\begin{itemize}
\item[{\ref{limsuppresentation}.a)}]
If $(x_n)_n, (y_n)_n \in {\mathcal T}^{+}$ are such that $x_n=y_n$ definitely, then $\overline{\ell}((x_n)_n)=\overline{\ell}((y_n)_n)$.

\item[{\ref{limsuppresentation}.b)}]
If $(x_n)_n, (y_n)_n \in {\mathcal T}^{+}$, then 
$\,\overline{\ell}(( x_n + y_n)_n) \leq_X \overline{\ell}((x_n)_n) + \overline{\ell}((y_n)_n)\,$ 
(\em subadditivity\rm).

\item[{\ref{limsuppresentation}.c)}]
If $(x_n)_n$, $(y_n)_n \in {\mathcal T}^{+}$ and $x_n \leq_X y_n$ definitely, then $\overline{\ell}
((x_n)_n) \leq_X \overline{\ell}((y_n)_n)$ (\em monotonicity\rm).

\item[{\ref{limsuppresentation}.d)}]
If a sequence $(x_n)_n \in {\mathcal T}^{+} \cap \mathcal{S}$, then   ${\overline{\ell}}((x_n)_n) =\ell((x_n)_n)$.

\item[{\ref{limsuppresentation}.e)}]
If a sequence $(x_n)_n \in {\mathcal T}^{+}$ is such that  ${\overline{\ell}}((x_n)_n) =0$, then $(x_n)_n \in \mathcal{S}$ and \mbox{$\ell((x_n)_n)=0$.}
\end{itemize}
\end{axiom}

A \emph{filter} on $\mathbb{N}$ is a family $\mathcal{F}$ of subsets of ${\mathbb{N}}$ such that $\emptyset \not \in {\mathcal F}$,
$A \cap B \in {\mathcal F}$ whenever $A$, $B \in \mathcal{F}$, and for every $A \in \mathcal{F}$ and $B \subset \mathbb{N}$
 with
%eliminati gli ultimi vecchi red
$B \supset A$, it is $B\in \mathcal{F}$.
We denote by ${\mathcal F}_{\text{cofin}}$ the filter of all cofinite subsets of  ${\mathbb N}$.
We say that a filter on ${\mathbb N}$  is \textit{free} iff  it contains ${\mathcal F}_{\text{cofin}}$.
An example of free filter is the filter of all subsets of $\mathbb{N}$ whose asymptotic density is equal to one  (see, e.g., \cite{BDBOOK}).

\begin{remark}\label{Acontinuity} \rm
 Some examples 
of convergences satisfying Axioms \ref{convergenze} are:
if ${\mathcal F}$ is a free filter,
 the \emph{relative uniform filter $((r{\mathcal F})$-$)$convergence} and the \emph{order filter $((o{\mathcal F})$-$)$convergence}
(see \cite[Definition 1.1]{BCSVITALI});
the \emph{almost convergence} and the \emph{Ces\`{a}ro convergence} (see, e.g., \cite{BDBOOK, BORSIKSALAT}).
\\
We recall that a sequence $(x_n)_n$ in $\mathbf{X}$ is
\emph{relatively uniformly filter $((r{\mathcal F})$-$)$con\-ver\-gent} to $x \in \mathbf{X}$,
iff there exists $u \in \mathbf{X}^+ \setminus \{0 \}$ such that
$\{ n \in \mathbb{N}: \vert   x_n - x\vert    \leq_X  \varepsilon \, u \} \in {\mathcal F}$ for each $\varepsilon \in \mathbb{R}^+$;
\emph{order filter $((o{\mathcal F})$-$)$convergent} to $x \in \mathbf{X}$,
iff there exists an $(o)$-sequence $(\sigma_l)_l$ with the property that, for all $l \in \mathbb{N}$, 
%\begin{eqnarray}\label{o-conv}
$\{ n \in \mathbb{N}: \vert   x_n - x\vert    \leq_X \sigma_l \} \in {\mathcal F}$.
If ${\mathcal F}$ is the filter ${\mathcal F}_{\rm{cofin}}$
of all subsets of ${\mathbb{N}}$
whose complement is finite, then the 
$(r{\mathcal F}_{\rm{cofin}})$- and
$(o{\mathcal F}_{\rm{cofin}})$-convergence coincide with the usual $(r)$- and
$(o)$-convergence, respectively (see, e.g., \cite{BDBOOK}).
\\ 
A sequence $(x_n)_n$ in $\mathbf{X}$ is  $(r)$- ( resp., $(o)$-) 
\emph{almost convergent} to $x \in \mathbf{X}$, iff 
\[(r)\text{-} ( \text{resp.,} (o)\text{-}) \lim_n 
\left(\bigvee_{m\in \mathbb{N}} 
\left\vert    \frac{x_{m+1} + x_{m+2} +\ldots + 
x_{m+n}}{n}- x \right\vert 
   \right) =0;\]
$(r)$- ( resp., \emph{order} or $(o)$-)
\emph{Ces\`{a}ro convergent} to $x \in \mathbf{X}$, iff 
\[(r)\text{-} ( \text{resp.,} (o)\text{-}) \lim_n \frac{x_1 + x_2 +\ldots + 
x_n}{n}= x.
\]
\\
Observe that $(r{\mathcal F})$-convergence implies $(o{\mathcal F})$-convergence, and in general  they do not coincide. However, when $\mathbf{X}=\mathbb{R}$,
these  convergences
are equivalent. In this case, we will  denote both convergences by \emph{$({\mathcal F})$-convergence}.\\
%Moreover, 
 Note that, in general,
%even when $\mathbf{X}=\mathbb{R}$, 
the usual convergence is strictly stronger than almost convergence, and almost convergence is strictly 
stronger than the Ces\`{a}ro one. 
Moreover, observe that almost and Ces\`{a}ro 
 convergences
are not equivalent to $({\mathcal F})$-convergence 
for \emph{any} free filter. Indeed, if a function  $\phi:\mathbb{R} \to \mathbb{R}$ is  sequentially continuous at $0$ with respect to
almost (resp., Ces\`{a}ro) convergence
(that is, $\phi$ maps sequences almost (resp., Ces\`{a}ro) convergent to $0$ into sequences almost (resp., Ces\`{a}ro) convergent to $0$), then
$\phi$ is linear, but sequential continuity with respect to $({\mathcal F})$-convergence   coincides with usual continuity (see, e.g.,  \cite{bbdmkorovkin, BORSIKSALAT, KSW}).
\end{remark}
\begin{example}\label{limsupoperators}
Now we give some examples of ``limsup''-type operators 
 satisfying Axioms \ref{limsuppresentation}
 and some related properties.
\begin{itemize}
\item[\ref{limsupoperators}.a)] 
 The  \emph{order limit superior}  of  a sequence $(x_n)_n$ in ${\mathbf{X}}^{+}$ is 
the element of $\overline{\mathbf{X}}^{+}$  defined by
$\displaystyle{
(o)\limsup_n x_n=\wedge_{m=1}^{\infty} (\vee_{n \geq m} \, x_n)}$
 (see, e.g., \cite{VULIKH}).

\item[\ref{limsupoperators}.b)]
 Let $\mathbf{X}=\mathbb{R}$ and ${\mathcal F}$ be any fixed free filter  on $\mathbb{N}$.
Given any sequence $(x_n)_n$ in $\mathbb{R}^{+}_0$, we call  \emph{$({\mathcal F})$-limit superior} 
of $(x_n)_n$ (shortly, $(\mathcal {F})\limsup_n x_n$) the element of $\overline{\mathbb{R}}^{+}_0$ defined by
$\displaystyle{
(\mathcal{F})\limsup_n x_n= \inf_{F \in \mathcal{F}} (\sup_{n \in F}\, x_n)}$
(see, e.g., \cite{BC20092, DEMIRCILIMSUP}).

\item[\ref{limsupoperators}.c)] 
Given any sequence $(x_n)_n$ in ${\mathbf{X}}^{+}$, we call 
\emph{order filter limit superior}  the element of $\overline{\mathbf{X}}^{+}$ defined by
$\displaystyle{(o\mathcal{F})\limsup_n x_n= \bigwedge_{F \in \mathcal{F}}  \Bigl(\bigvee_{n \in F} x_n \Bigr)}$.
The order filter limit superior satisfies Axioms \ref{limsuppresentation}.a) - \ref{limsuppresentation}.d) (see, e.g., \cite{BC20092}).
Observe that, in the setting of vector lattices, when $\mathbf{X} \neq \mathbb{R}$, the problem of finding suitable ``limsup''-type operators
 satisfying Axiom \ref{limsuppresentation}.e) is still open, though some positive partial answers are found in \cite{BC20092}.
 For example, if $T$ is any nonempty set and 
${\mathbf{X}}={\mathbb{R}}^T$, the space of all real-valued
functions defined on $T$, or if $R$ 
is a Dedekind complete vector lattice and ${\mathbf{X}}=
R^{\sim}$ (resp., %${\mathbf{X}}=
$R^{\times}$) is the space of all linear order bounded 
(resp., order continuous) functionals, 
then the lattice and 
"component-wise" suprema and infima coincide
(see, e.g., \cite{POKKUTO}). From this,
since Axiom \ref{limsuppresentation}.e) is satisfied when 
${\mathbf{X}}={\mathbb{R}}$,
then it holds also when ${\mathbf{X}}={\mathbb{R}}^T$,
%${\mathbf{X}}=
$R^{\sim}$ or 
%${\mathbf{X}}=
$R^{\times}$. Moreover, observe that the above argument
can be applied even if ${\mathbf{X}}$ is a vector lattice whose
elements are equivalence classes of functions, and in which the lattice 
suprema/infima coincide with the respective classes of equivalence
to which the 
pointwise suprema/infima belong. This is the case of 
the so-called "Dedekind complete $\rho$-spaces", whose 
examples are the spaces
$L^p(\Omega, \Sigma, \nu)$, where $0 \leq p \leq \infty$,
$(\Omega, \Sigma, \nu)$ is a measure space and
$\nu: \Sigma \to \mathbb{R}^{+}_0$ is a 
$\sigma$-additive and $\sigma$-finite measure,
which do not have a strong order unit 
and in which the order convergence does not have
a topological nature (see, e.g., 
\cite{POKKUTO, VULIKH}). These spaces can be viewed also 
"directly" as spaces of classes of equivalences of 
real-valued functions defined on $\Omega$ 
up to $\nu$-null sets,
in which the suprema and infima 
have the aforementioned property.

%===========================

%==========================
\item[\ref{limsupoperators}.d)] 
Suppose that ${\mathbf{X}}$ has a strong order unit $e$,  fix any free filter ${\mathcal F}$ on $\mathbb{N}$ and let us endow ${\mathbf{X}}$ with
$(r{\mathcal F})$-convergence  (see, e.g., 
\cite{BCSVITALI}). Then, it is not difficult to see that
a sequence
%We claim that a sequence 
$(x_n)_n$ in $\mathbf{X}$ is $(r{\mathcal F})$-convergent to $x \in \mathbf{X}$ if and only if
%\begin{eqnarray}\label{eFc}
$\{ n \in \mathbb{N}: \vert   x_n - x\vert    \leq_X \varepsilon \, e\} \in {\mathcal F}$
%\end{eqnarray}
for every $\varepsilon \in \mathbb{R}^+$.

%======================
Now we  endow ${\mathbf{X}}$ 
with the norm $\|  \cdot\| _e$, defined by
\begin{eqnarray}\label{latticenorm}
\| x\|  _e=\inf \{\varepsilon \in 
\mathbb{R}^+:
 \vert   x\vert   \leq_X \varepsilon \, e\}
\end{eqnarray}
(see, e.g., 
\cite{bodivi, CSorder, CS2015, MEYER, SCHAEFER}). 
 Observe
that $\|  \cdot\|  _e$ is \emph{monotone}, that is
$\| x\|  _e \leq \|  y\|_e$ whenever $x,y \in \mathbf{X}$ and 
$\vert   x\vert    \leq \vert   y\vert   $. Moreover, thanks to the properties of the
infimum, it is not difficult to see that
%============
%It is not difficult to check that $\\vert   \cdot\\vert   _e$ is a norm,  that
\begin{eqnarray}\label{prime}
\vert   x \vert    \leq \|  x \|_e \, e ,
\end{eqnarray}
\begin{eqnarray}\label{doubleimplication}
\vert   x\vert    \leq_X \alpha \, e \quad \Longleftrightarrow \quad  \|  x\|_e \leq \alpha,
\end{eqnarray}
for all $\alpha \in
\mathbb{R}^+ $ and $x \in {\mathbf{X}}$
(see, e.g., \cite[\S 7.4 (2)]{VULIKH},
\cite[Proposition 1.2.13]{MEYER}), and
\begin{eqnarray}\label{triple}
\|  \alpha \, e \|_e = \vert   \alpha \vert    \quad \text{for  all    }
\alpha \in \mathbb{R}.
\end{eqnarray}
%====================
Then,
a sequence in ${\mathbf{X}}$ $(r{\mathcal F})$-converges  if and only if it $({\mathcal F})$-converges with respect to the norm $\|  \cdot\| _e$.
Thus, analogously as in \cite{BC20092,DEMIRCILIMSUP},
the operator
${\overline{\ell}}((x_n)_n)=
\displaystyle{({\mathcal F})\limsup_n \|  x_n\| _{e}}$
%\end{eqnarray}
satisfies Axioms \ref{limsuppresentation}, because  they are fulfilled by the usual 
filter limit superior in ${\mathbb{R}}$, and taking into account the properties of the norm. 

\item[\ref{limsupoperators}.e)] 
Analogously as in Remark \ref{Acontinuity},  it is possible to associate in a natural way an ``order limsup''-type operator
%one can see that
 to $(o)$-almost and $(o)$-Ces\`{a}ro convergences
 defined in Remark
\ref{Acontinuity}, 
%defined in Examples \ref{example}
satisfying Axioms \ref{limsuppresentation}.
\item[{\ref{limsupoperators}.f)}] Note that, if $\mathbb{R}$ is
endowed with any convergence 
$\ell_{\mathbb{R}}$ satisfying 
Axioms \ref{convergenze} and 
whose corresponding``limit superior'' operator
$\bar{\ell}_{\mathbb{R}}$ fulfils
Axioms \ref{limsuppresentation},
and $(x_n)_n$ is any sequence 
of non-negative real numbers 
such that
$\displaystyle{\lim_n x_n=0}$
%$(x_n)_n$ is an $(o)$-sequence 
in the usual sense, then $\ell_{\mathbb{R}}(x_n)_n=0$.
Indeed, fixed arbitrarily $\varepsilon \in {\mathbb{R}}^+$,
there is a positive integer $n_0$ with $x_n \leq \varepsilon$ 
whenever $n \geq n_0$. From this and Axioms \ref{convergenze}.c),
\ref{limsuppresentation}.c) and \ref{limsuppresentation}.d) we deduce that 
$\bar{\ell}_{\mathbb{R}}(x_n)_n \leq \varepsilon$. By the 
arbitrariness of $\varepsilon$ and Axiom \ref{limsuppresentation}.e) 
it follows that $\ell_{\mathbb{R}}(x_n)_n=0$.
\end{itemize}
\end{example}
%==================================================
%==================================================
\section{The integral with respect to abstract convergences}\label{product}
In this section we give the construction of an abstract integral for vector lattice-valued 
functions with respect to (possibly infinite) vector lattice-valued measures, 
extending the integrals presented in \cite{BC2009,BCSVITALI}.
Let $\mathbf{X}_1$, $\mathbf{X}_2$, $\mathbf{X}$ be  three Dedekind complete vector lattices. 
\begin{assumption}\label{ass}
\rm
We say that  $(\mathbf{X}_1,\mathbf{X}_2,\mathbf{X})$ is a \emph{product triple}
iff a ``product operation'' $\cdot :\mathbf{X}_1 \times \mathbf{X}_2 \to \mathbf{X}$ is defined, 
satisfying the following conditions \cite[Assumption 2.2]{BC2009}:
\begin{itemize}
\item[{\ref{ass}.1)}]
$ (x_1+y_1) \cdot x_2=  x_1 \cdot x_2 +y_1 \cdot x_{2}$,
\item[{\ref{ass}.2)}]
$ x_1 \cdot (x_2+y_2) = x_1 \cdot x_2+x_1  \cdot y_2$,
\item[{\ref{ass}.3)}]
$[x_1 \geq_{X_1} y_1$, $x_2 \geq_{X_2} 0] \Rightarrow [x_{1} \cdot x_{2} \geq_{X} y_{1} \cdot x_{2}]$, 
\item[{\ref{ass}.4)}]
$[x_1 \geq_{X_1} 0$, $x_2 \geq_{X_2} y_2] \Rightarrow [x_{1} \cdot x_{2} \geq_{X} x_{1} \cdot y_{2}]$
for every $x_j$, $y_j \in {\mathbf{X}}_j$, $j=1,2$;
\item[{\ref{ass}.5)}]
 if $(x_n)_n$ is an $(o)$-sequence in  ${\mathbf{X}}_1$ and $y \in {\mathbf{X}}_2^{+}$, then 
$(x_n \cdot y )_n$ is an $(o)$-sequence in ${\mathbf{X}}$;
\item[{\ref{ass}.6)}]
if $x \in {\mathbf{X}}_1^{+}$ and $(y_n)_n$ is an $(o)$-sequence in ${\mathbf{X}}_2$, then 
$(x \cdot y_n)_n$ is an $(o)$-sequence in ${\mathbf{X}}$ (see, e.g., \cite{KAWABE, KAWABE2}),
\end{itemize}
and if   $\mathbf{X}_1$, $\mathbf{X}_2$, $\mathbf{X}$ are endowed  with three convergences, $\ell_1 $,
$\ell_2$, $\ell$, respectively, each of which satisfying 
 Axioms \ref{convergenze} and
the following ``compatibility'' conditions:
\begin{itemize}\label{compatibility}
\item[{\ref{ass}.7)}] 
if $(x_n)_n$ is a sequence in  ${\mathbf{X}}_1$ with $\ell_1((x_n)_n) = 0$ and $y \in {\mathbf{X}}_2$, then 
${\ell}((x_n \cdot y )_n)=0$;
\item[{\ref{ass}.8)}] 
if $x \in {\mathbf{X}}_1$ and $(y_n)_n$ is a sequence in  ${\mathbf{X}}_2$ with $\ell_2 ((y_n)_n) = 0$,
then ${\ell}((x \cdot y_n )_n)=0$;
\item[{\ref{ass}.9)}] 
if $(x_n)_n$ is a sequence in ${\mathbf{X}}_1^+$ with ${\overline{\ell}}_1((x_n)_n) =x$ and $y \in {\mathbf{X}}_2^+$, then 
${\overline{\ell}}((x_n \cdot y )_n)=x \cdot y$;
\item[{\ref{ass}.10)}] 
if $x \in {\mathbf{X}}_1^+$ and $(y_n)_n$ is a sequence in  ${\mathbf{X}}_2^+$ with ${\overline{\ell}}_2 ((y_n)_n) = y$,
then ${\overline{\ell}}((x \cdot y_n )_n)=x \cdot y$.
\end{itemize}
\end{assumption}
A Dedekind complete vector lattice ${\mathbf{X}}$ is said to be a  \emph{product algebra} iff $(\mathbf{X},\mathbf{X},\mathbf{X})$
is a product triple. We endow the real line $\mathbb{R}$ with a convergence $\ell_{\mathbb{R}}$ satisfying 
 Axioms \ref{convergenze} and assume that $\mathbf{X}_1$, $\mathbf{X}_2$, $\mathbf{X}$
%\mg{\tiny ??  diciamo un "Bochner generalizzato", ant.}
are product algebras and  $(\mathbb{R},\mathbf{X},\mathbf{X})$,  $(\mathbb{R},\mathbf{X}_j,\mathbf{X}_j)$, $j=1,2$, are product  triples. 
 Moreover, suppose that there exist 
two product algebras
$\mathbf{X}_1^{\prime}$
and $\mathbf{X}_1^{\prime\prime}$,
 endowed with respective convergences $\ell_1^{\prime}$ and 
$\ell_1^{\prime \prime}$ satisfying
Axioms \ref{convergenze},  such that
$(\mathbf{X}_1^{\prime}, \mathbf{X}_1^{\prime\prime}, 
\mathbf{X}_1)$
is a product triple.
\\

Let $G$ be any  (possibly infinite) nonempty set, ${\mathcal P}(G)$ be the class of all subsets of $G$, ${\mathcal A} \subset {\mathcal P}(G)$ be an algebra,
$\mu:{\mathcal A} \to  \overline{\mathbf{X}}^{+}_2$ be a finitely additive measure. 
We say that $\mu$ is \emph{$\sigma$-finite} iff there is a 
%(increasing) 
sequence  $(B_n)_n$   in ${\mathcal A }$, with  
%\begin{eqnarray}\label{sigmafiniteness}
$\mu(B_n) \in {\mathbf{X}}^{+}_2$, $n \in \mathbb{N}$,  and
%\, \,  \& \,\,   
$\displaystyle{\bigcup_{n\in \mathbb{N}} B_n=G}$. 
  Note that, without loss of generality, the $B_n$'s
can be taken increasing or disjoint. 
\\
%\end{eqnarray}
Now we introduce an integral for $\mathbf{X}_1$-valued functions with respect to a finitely additive and $\sigma$-finite 
$\overline{\mathbf{X}}^+_2$-valued measure $\mu$, related to convergences $\ell_1$, $\ell_2$, $\ell$,
satisfying Axioms \ref{convergenze}. 
The integral will be an element of $\mathbf{X}$.
Similar constructions in the vector lattice setting were given in \cite{BC2009},
where $\mu$ is finite, and in \cite{BCSVITALI}, where $\mu$ is possibly infinite and 
$\mathbf{X}_1$, $\mathbf{X}$ are endowed  with 
two strong order units $e_1$, $e$, respectively. \\

We consider a sequence of functions of finite range and vanishing outside of a set of finite measure $\mu$, as follows.
A function $f:G \to \mathbf{X}_1$ is said to be \emph{simple} iff   it can be expressed as 
%\begin{eqnarray}\label{simple}
$f=\sum_{j=1}^r c_j \, \chi_{A_j},$
%\end{eqnarray}
where $r \in \mathbb{N}$, $c_j \in {\mathbf{X}}_1$, $A_j \in {\mathcal A}$, $\mu(A_j) \in {\mathbf{X}}^{+}_2$ and
$\chi_{A_j}$   is  the characteristic function of the set   $A_j$.
We denote by $\mathscr{S}$
the set  of all simple functions.
For each $A \in {\mathcal A}$ and $f \in \mathscr{S}$, where $f$ is as before, set 
\begin{eqnarray*}\label{intsimple}
\int_A f(g) \, d\mu(g)= \sum_{j=1}^r c_j \, \mu(A \cap A_j).
\end{eqnarray*}
The integral %in (\ref{intsimple})
of a simple function  is a linear and monotone functional, 
and does not depend on the choice of the representation of $f$.  
Now we present the concepts of uniform convergence and convergence in measure.
\begin{definition}\label{measureuniform} \rm 
Let $A \in {\mathcal A}$ and $(f_n)_n$   be
a sequence of functions  in $\mathbf{X}_1^G$ :
\begin{itemize}
\item $(f_n)_n$ is said to be \emph{uniformly convergent} to $f\in \mathbf{X}_1^G$ on $A$
iff  \[\displaystyle{\ell_1 \Bigl(\Bigl( \bigvee_{g \in A} \, \vert f_n(g)-f(g) \vert\Bigr)_n\Bigr)=0};\]
%eliminati gli ultimi vecchi red
\item 
$(f_n)_n$ \em converges in $\mu$-measure \rm to $f\in \mathbf{X}_1^G$ on $A$ iff there is
a sequence $(A_n)_n$ in $\mathcal{A}$, such that $\ell_2((\mu(A \cap A_n))_n)=0$ and 
$\displaystyle{ \ell_1 \Bigl(\Bigl(\bigvee_{g \in A \setminus A_n} \vert f_n(g)-f(g) \vert\Bigr)_n\Bigr)=0}$.
\end{itemize}
\end{definition} 
Observe that uniform convergence implies convergence in measure, and that, when ${\mathbf{X}}_1={\mathbf{X}}_2=\mathbb{R}$ and the involved
convergence is the usual one, the convergences  in Definition  
\ref{measureuniform} 
are equivalent to the classical ones (see, e.g., \cite[Remark 3.3]{BC2009}).
The following result extends \cite[Proposition 3.4]{BC2009} when $\mu$ is not necessarily finite.

\begin{proposition}\label{Mconv}
 Let $A \in {\mathcal A}$. If $(f_n)_n$, $(h_n)_n \in {\mathbf{X}}_1^G$ converge in measure to
$f,h$ on $A$, respectively, and $c \in {\mathbb{R}}$, then $(f_n+h_n)_n$, $(f_n\vee h_n)_n$, $(f_n\wedge
h_n)_n$, $(\vert f_n \vert)_n$ converge in measure on $A$ to $f+h$, $f \vee h$, $f \wedge h
$, $\vert f \vert$, respectively, and $(c \, f_n)_n$ converges in $\mu$-measure to $c \, f$ on $A$.
The same results hold for uniform convergence.
\end{proposition}

\begin{proof} 
We consider only convergence in measure, since  the case of uniform convergence is analogous.
We begin with proving the convergence in measure of  the sequence $(f_n+h_n)_n$ to $f+h$ on $A$.
Let $(A_n)_n$ and $(D_n)_n$ be two sequences in  ${\mathcal A}$, related with the convergence in measure  of $(f_n)_n$ 
and $(h_n)_n$ to $f$ and $h$ on $A$, respectively,  and set $E_n=A_n \cup D_n$, $n \in \mathbb{N}$.
Note that $0 \leq_{X_2} \mu(E_n) \leq_{X_2} \mu(A_n)+
\mu(D_n)$, and thus from Axioms  
\ref{convergenze}.a) and \ref{convergenze}.e)  we deduce
\begin{eqnarray}\label{complementunion}
 \ell_2((\mu(E_n))_n)=0. 
\end{eqnarray}
Moreover, for each $n \in \mathbb{N}$ we have
\begin{eqnarray*}\label{sussmayrsum} 
0 &\leq_{X_1}& \Bigl( \bigvee_{g \in A \setminus E_n} \, \nonumber
\vert f_n(g)+ h_n(g)-f(g)- h(g) \vert \Bigr) \leq_{X_1}  \\ &\leq_{X_1}& \Bigl( \bigvee_{g \in A \setminus E_n} \,
\vert f_n(g)-f(g) \vert  \Bigr) + 
 \Bigl( \bigvee_{g \in A \setminus E_n} \, \vert h_n(g)-f(g)\vert \Bigr)  \leq_{X_1} \\ &\leq_{X_1}&
\Bigl( \bigvee_{g \in A \setminus A_n} \, \vert f_n(g)-f(g)\vert \Bigr)  + 
\Bigl( \bigvee_{g \in A \setminus D_n} \, \vert h_n(g)-f(g)\vert \Bigr) . 
\nonumber
\end{eqnarray*}
From the previous inequality
%(\ref{sussmayrsum})
 and  Axioms
\ref{convergenze} it follows that the sequence  $(f_n+h_n)_n$ converges in measure to $f+g$ on $A$. \\
Now we show that the sequence  $(f_n\vee h_n)_n$ converges in measure to $f \vee g$ on $A$.
Taking into account the Birkhoff inequalities (see \cite[Theorem 12.4 (ii)]{LZ}), we get
\begin{eqnarray}\label{sussmayrsup} 
&&0 \leq_{X_1} \Bigl( \bigvee_{g \in A \setminus E_n} \,  \vert  f_n(g) \vee h_n(g)- f(g) \vee h(g) \vert  \Bigr)  \leq_{X_1}  \nonumber  \\
 &&\leq_{X_1} 
\Bigl( \bigvee_{g \in A \setminus E_n} \,  \vert  f_n(g) \vee h_n(g)- f_n(g) \vee h(g) \vert \Bigr)  + \\ 
&&+ \nonumber
\Bigl( \bigvee_{g \in A \setminus E_n} \, \vert  f_n(g) \vee h(g)- f(g) \vee h(g) \vert  \Bigr)  \leq_{X_1} 
\\ &&\leq_{X_1} \nonumber
 \Bigl( \bigvee_{g \in A \setminus A_n} \,\vert f_n(g)-f(g) \vert \Bigr)  + 
\Bigl( \bigvee_{g \in A \setminus D_n} \,\vert h_n(g)-f(g)\vert \Bigr) .\nonumber
\end{eqnarray}
From (\ref{sussmayrsup}) and Axioms \ref{convergenze} again it follows that the sequence  $(f_n \vee h_n)_n$ converges in measure to $f \vee h$ on $A$.
The proof of the convergence in measure of the sequence  $(f_n \wedge h_n)_n$ to $f \wedge h$ is analogous. \\
The convergence in $\mu$-measure of the sequence $(\vert f_n\vert )_n$   to  $\vert f \vert $ follows from the previous results, taking into account 
%thatfor each $x \in {\mathbf{X}}_1$ it is $\vert x \vert=(x \vee 0)+ ((-x) \vee 0)$.
the definition of $\vert x \vert$.  
The proof of the convergence in $\mu$-measure of the sequence  $(c \, f_n)_n$ to $c \, f$ is analogous to the previous ones.
\end{proof}

\begin{definition}\label{L1}
\rm We say that a sequence $(f_n)_n$ in  $\mathscr{S}$ \em converges in $L^1$ \rm to $f \in \mathscr{S}$ iff 
\begin{eqnarray*}
 \ell \left( \left(\int_G \vert f_n(g)-f(g)\vert  \, d\mu(g) \right)_n \, \right)=0. 
\end{eqnarray*}
\end{definition}

\begin{definition}\label{eac}
\rm We say that the integrals  of a sequence of functions $(f_n)_n$ in $\mathscr{S}$
%$\mathbf{X}_1^G$ 
are $\mu$-\emph{equiabsolutely continuous} iff the following two properties hold:
\begin{itemize}
\item[{\rm \ref{eac}.1)}] 
$\displaystyle{\ell \left( \left(\int_{A_n}\,\vert f_n(g)\vert  \, d\mu(g)\right)_n \right)=0}$\,\,  if \,\, $\ell_2((\mu( A_n))_n)=0$;
%%ATTENZIONE! %%PUNTO CRUCIALE!!!!
\item[{\rm \ref{eac}.2)}]
 there is an increasing sequence  $(B_m)_m$ in ${\mathcal A}$ with
$\mu(B_m)\in {\mathbf{X}}_2 $ for all $m \in \mathbb{N}$, and
\begin{eqnarray*}\label{eaceac2}
\ell\Bigl(\Bigl(\overline{\ell}\Bigl( \Bigl(\int_{G \setminus B_m} \vert f_n(g)\vert  \, d\mu(g) \Bigr)_n   \Bigr) \Bigr)_m \Bigr)=0.
\end{eqnarray*}
\end{itemize}
\end{definition}

The next result extends \cite[Theorem 3.7]{BC2009}   when  $\mu$ is  not necessarily finite.

\begin{theorem}\label{tredieci}
Let $f_n \in \mathscr{S}$, $n \in \mathbb{N}$. If $(f_n)_n$ converges in measure to $f \equiv 0$ on every set $A \in {\mathcal A}$ with $\mu(A) \in
{\mathbf{X}}_2$ and their integrals are $\mu$-equiabsolutely continuous, then $(f_n)_n$ converges in $L^1$ to $f \equiv 0$.
\end{theorem} 

\begin{proof}
Let $(B_m)_m$ be a sequence related to property 
\ref{eac}.2) of the  $\mu$-equiabsolute continuity 
of the integrals of the $f_n$'s. By the convergence
in measure of $(f_n)_n$ to $0$ on each $B_m$, there is  a double sequence $(A_{m,n})_n$ in ${\mathcal A}$ with
$\ell_2((\mu(B_m \cap A_{m,n}))_n)=0$ and 
$\displaystyle{\ell\Bigl(\Bigl(\bigvee_{g \in B_m \setminus A_{m,n}} \vert f_n(g)\vert \Bigr)_n\Bigr)=0}$
for each $m \in \mathbb{N}$. For every fixed  $m$, $n \in \mathbb{N}$ it is
\begin{eqnarray*}\label{firststep}
	0 & \leq_{X} & %I_n^{(0)} =
 \int_G\vert f_n(g)\vert  \, d\mu(g) \leq_{X}
	\int_{G\setminus B_m}\vert f_n(g)\vert  \, d\mu(g) + 
	\int_{B_m \cap A_{m,n}}\vert f_n(g)\vert  \, d\mu(g)+ \nonumber \\ &+&
	\int_{B_m \setminus A_{m,n}}\vert f_n(g)\vert  \, d\mu(g).%=I_n^{(1)} + I_n^{(2)}+I_n^{(3)}.
\end{eqnarray*}
By \ref{eac}.2), we get 
\begin{eqnarray}\label{n1}
\ell \Bigl( \Bigl( \overline{\ell} \Bigl(\Bigl(\int_{G\setminus B_m}\vert f_n(g)\vert  \, d\mu(g) \Bigr)_n\Bigr)\Bigr)_m\Bigr)=0.
\end{eqnarray}
From \ref{eac}.1) we obtain
\begin{eqnarray*}\label{prelimi}
\ell\Bigl(\Bigl(\int_{B_m \cap A_{m,n}}\vert f_n(g)\vert  \, d\mu(g)\Bigr)_n\Bigr)=0\quad \text{ for  every   } m \in \mathbb{N}.
\end{eqnarray*}
Taking into account 
Axioms \ref{convergenze} and \ref{limsuppresentation}.d), %from (\ref{prelimi})
then we deduce
\begin{eqnarray}\label{n2}
\ell \Bigl( \Bigl( \overline{\ell} \Bigl(\Bigl(\int_{B_m \cap A_{m,n}}\vert f_n(g)\vert  \, d\mu(g) \Bigr)_n\Bigr)\Bigr)_m\Bigr)=0.
\end{eqnarray}

Furthermore, for each $m \in \mathbb{N}$ it is
\begin{eqnarray*}\label{convergenceinmeasure}
	0 \leq_{X}  \int_{B_m \setminus A_{m,n}}\vert f_n(g)\vert  \, d\mu(g) &\leq_{X}& \Bigl(\bigvee_{g \in B_m \setminus A_{n,m}} \vert f_n(g)\vert \Bigr)
	\mu (B_m \setminus A_{n,m}) \leq_{X} \nonumber
	\\ &\leq_{X}& \Bigl(\bigvee_{g \in B_m \setminus A_{n,m}} \vert f_n(g)\vert \Bigr) \mu(B_m).
\end{eqnarray*}  
Since $\mu(B_m) \in \mathbf{X}_2$, from the previous inequality, 
convergence in measure on $B_m$
and  Axioms \ref{convergenze} we   get
\begin{eqnarray}\label{pren3}
\ell\Bigl(\Bigl(\int_{B_m \setminus  A_{m,n}}\vert f_n(g)\vert  \, d\mu(g)\Bigr)_n\Bigr)=0\quad \text{ for  each   } m \in \mathbb{N},
\end{eqnarray}
and hence, analogously as in (\ref{n2}), from (\ref{pren3}) we
  obtain
\begin{eqnarray}\label{n3}
\ell \Bigl( \Bigl( \overline{\ell} \Bigl(\Bigl(\int_{B_m \setminus A_{m,n}}\vert f_n(g)\vert  \, d\mu(g) \Bigr)_n\Bigr)\Bigr)_m\Bigr)=0.
\end{eqnarray}
From (\ref{n1}), (\ref{n2}), (\ref{n3})
 and Axioms \ref{convergenze}, \ref{limsuppresentation} it follows that 
\begin{eqnarray*}\label{conclu}
	0 &\leq_{X}&  \overline{\ell}\Bigl(\Bigl(\int_G\vert f_n(g)\vert  \, d\mu(g) \Bigr)\Bigr)_n\Bigr) = 
	\ell \Bigl( \Bigl( \overline{\ell} \Bigl(\Bigl(\int_G \vert f_n(g)\vert  \, d\mu(g) \Bigr)_n\Bigr)\Bigr)_m \Bigr) \leq_{X} 	\nonumber \\ 
	&\leq_{X}& 
	 \ell \Bigl( \Bigl( \overline{\ell} 
\Bigl(\Bigl(\int_{G\setminus B_m} \vert f_n(g)\vert  d\mu(g) 
\Bigr)_n\Bigr)\Bigr)_m\Bigr) + %\mg{$I_1$}
 \nonumber 
	\ell \Bigl( \Bigl( \ell \Bigl(\Bigl(\int_{B_m \cap A_{m,n}} \vert f_n(g)\vert   d\mu(g) \Bigr)_n\Bigr)\Bigr)_m\Bigr) +   %I_2
\\	&+& \nonumber	\ell \Bigl( \Bigl( \overline{\ell} \Bigl(\Bigl(\int_{B_m \setminus A_{m,n}}\vert f_n(g)\vert  \, d\mu(g) \Bigr)_n \Bigr)\Bigr)_m\Bigr)=0. %I_3
\end{eqnarray*}
From this %(\ref{conclu}) 
and  Axiom
\ref{limsuppresentation}.e) we obtain
$\displaystyle{
{\ell}
\Bigl(\Bigl(\int_G\vert  f_n(g)\vert  \, d\mu(g)\Bigr)\Bigr)_n
\Bigr)=0}$, that is the convergence in $L^1$ to $0$
of the sequence $(f_n)_n$.
\end{proof}
Now we turn to the construction of the  integral. 

\begin{definition} \label{defining} 
\rm Let $f \in {\mathbf{X}}_1^G$.
A sequence $(f_n)_n$ in $\mathscr{S}$ is said to be  \emph{defining for $f$} iff it converges in $\mu$-measure to $f$  on every set $A \in {\mathcal A}$ with $\mu(A) \in
\mathbf{X}_2$ and their  integrals are $\mu$-equiabsolutely continuous.
\end{definition}

\begin{definition}\label{integrabilita} \rm 
A positive function $f \in {\mathbf{X}}_1^G$ is said to be \em integrable \rm on $G$ iff there exist a defining sequence $(f_n)_n$ 
for $f$ and a map $l:{\mathcal A} \to \mathbf{X}$, with
\begin{eqnarray}\label{l} 
	\ell \Bigl( \Bigl(\bigvee_{A\in {\mathcal A}}\Bigl\vert \int_A f_n(g) \, d\mu(g) - l(A)\Bigr\vert \Bigr)_n \Bigr)=0,
\end{eqnarray} 
and in this case we set 
\begin{eqnarray}\label{integral}
	\int_A f(g) \, d\mu(g):=l(A) \quad \text{ for  every  } A \in {\mathcal A}.
\end{eqnarray}
\end{definition}
Now we prove that the integral in (\ref{integral}) is well-defined, extending \cite[Proposition 3.11]{BC2009} and 
\cite[Theorem 3.5]{BCSVITALI}.

\begin{proposition}\label{welldefined} 
For every $A \in {\mathcal A}$, it is
	$l(A):= \ell \left( \left(\displaystyle{\int_A} \, f_n(g)\,d\mu(g) \right)_n \right)$,
and  $l(A)$
does not depend on the choice of the defining sequence.
\end{proposition}
\begin{proof} We   prove that the assertion 
is a consequence of Definition \ref{integrabilita} 
(formula (\ref{l})) 
and Axioms \ref{convergenze}.  For $j=1,2$, let 
$(f_n^{(j)})_n$ be two defining sequences for $f$, put
\begin{eqnarray*}\label{16bis}
	l^{(j)} (A) :=\ell \left( \left(\int_A \, f_n^{(j)} (g)\,d\mu(g)\right)_n \right)\quad  \text{for each  } \,\,  A \in {\mathcal A},
\end{eqnarray*} 
and set $q_n(g)=\vert f_n^{(1)}(g)-f_n^{(2)}(g) \vert$, $g \in G$, $n \in \mathbb{N}$.
%It is not difficult to check that
 The sequence $(q_n)_n$ converges in $\mu$-measure to $0$ on every set $A\in {\mathcal A}$ with $\mu(A) \in {\mathbf{X}}_2$.
%because $(f_n^j)_n$, $j=1,2$, converge in  $\mu$ measure to $f$.
Moreover, the integrals of $(f_n^{(j)})_n$, $j=1,2$, 
and thus also those of the $q_n$'s, are 
 $\mu$-equiabsolutely continuous. Thus, by Theorem \ref{tredieci}, the sequence
$(q_n)_n$ converges to $0$ in $L^1$. Moreover we get, for every $n \in \mathbb{N}$,
\begin{eqnarray}\label{l1l2} \nonumber
	0 &\leq_{X}& \vert l^{(1)}(A)-l^{(2)}(A)\vert \leq_{X} \left\vert\int_A f_n^{(1)}(g) d\mu(g)-l^{(1)}(A)\right\vert +  \\ &+&\left\vert l^{(2)}(A)- \int_A f_n^{(2)}(g) d\mu(g)\right\vert +
  \left\vert \int_A f_n^{(1)}(g) \, d\mu(g) -
	\int_A f_n^{(2)}(g) \, d\mu(g)\right\vert  \leq_{X} \\ &\leq_{X}& 
	\left\vert \int_A f_n^{(1)}(g) \, d\mu(g)-l^{(1)}(A)\right\vert +
 \left\vert l^{(2)}(A)- \int_A 
f_n^{(2)}(g) d\mu(g)\right\vert + \nonumber \int_G q_n(g) d\mu(g). \nonumber
\end{eqnarray} 
From %(\ref{16bis}), 
the definition of $l^{(j)}$,  formula
(\ref{l1l2}), the convergence in $L^1$ to $0$ of $(q_n)_n$ and 
 Axioms \ref{convergenze}.a), \ref{convergenze}.b),
 \ref{convergenze}.g) it follows that 
$\vert l^{(1)}(A)-l^{(2)}(A) \vert=0$,  namely $l^{(1)}(A)=l^{(2)}(A)$, for all $A \in {\mathcal A}$. 
\end{proof}

Now we define our integral for not necessarily positive functions.
\begin{definition}\label{integrabilitaa} \rm
A function $f: G \to {\mathbf{X}}_1$ is said to be \emph{integrable} on $G$ iff the functions $f^+(g)=f(g)\vee 0$, $f^-(g)=(-f(g))\vee 0$, $g \in G$, are integrable on $G$, 
and in this case we set
\[\int_A f(g) \, d\mu(g) = \int_A f^+(g) \, d\mu(g) - \int_A f^-(g) \, d\mu(g)\quad \text{for every } \,\,A \in {\mathcal A}.\]
\end{definition}

\begin{remark}\label{pf}
\begin{itemize}
\item[\ref{pf}.a)] %It is not difficult to see that
Since $\vert f(g) \vert =f^+(g)+f^-(g)$ for every $g \in G$, if $(f_n)_n$ and $(h_n)_n$ are two sequences
of functions defining for $f^+$ and $f^-$, respectively, then $(f_n+h_n)_n$ is a sequence of functions defining for $\vert f \vert $.
From this it follows that, if $f$ is integrable, then $\vert f \vert $ is integrable too, and
\[\int_A \vert f(g) \vert  \, d\mu(g) = \int_A f^+(g) \, d\mu(g) +  \int_A f^-(g) \, d\mu(g) \quad \text{for each  } \, A \in {\mathcal A}.\]
Moreover, note that
% it is not difficult to check that
the integral defined in \ref{integrabilitaa}  is a linear positive ${\mathbf{X}}$-valued functional.

\item[\ref{pf}.b)]
 Observe that, if $v_{1} \in \mathbf{X}_1$ and  $h:G \to \mathbb{R}$ is integrable, then the function 
$f:G  \to  \mathbf{X}_1$ defined by   $f(g)=(h(g)) \, v_{1}, \,\, g \in G$, is integrable, and 
\begin{eqnarray*}\label{inthf}
	\int_A f(g) \, d\mu(g)= \Bigl( \int_A h(g) \, d\mu(g) \Bigr) \, v_{1}\quad \text{for every  } \, A \in {\mathcal A}.
\end{eqnarray*}
Indeed,  taking into account {\ref{ass}.7)}, this
%(\ref{inthf})
follows from the fact that, is $(h_n)_n$ is a defining  sequence for $h$, then $(h_n \, v_{1})_n$ is a defining 
sequence for $f$.

\item[\ref{pf}.c)] 
If $\mu$ is finite, then the integrals  given in \cite{BC2009} and in  Definition \ref{integrabilitaa} coincide 
(indeed, it will be enough to take $B_m=G$  for each $m \in \mathbb{N})$. 

\item[\ref{pf}.d)] 
The notions of convergence in $L^1$ and $\mu$-equiabsolute continuity can be given also for 
sequences of integrable functions, analogously as in  Definitions \ref{L1} and \ref{eac}, respectively.

\item[\ref{pf}.e)] 
 Moreover, we observe that,  if $f$ and $h$ integrable
and $\alpha \in \mathbb{R}$, then $f+h$ and
$\alpha \, f$ are integrable too, thanks to the properties of the defining sequences and the linearity of $\ell$ (Axiom \ref{convergenze}.a)).

\item[\ref{pf}.f)]

If $f$ is non negative and integrable, then, by Proposition \ref{Mconv}, its integral is positive too. 
In fact, if $(f_n)_n$ $\mu$-converges to $f$, then $(f_n^+)_n=(f_n \vee 0)_n$ $\mu$-converges to $f \vee 0=f$.
In addition, 
$\vert f_n^+\vert  =f_n^+ \leq_{X_1} \vert f_n \vert$, for every $n \in \mathbb{N}$, and then we obtain the
$\mu$-equiabsolute continuity of the $f_n^+$ integrals, which ensures that $(f_n^+)_n$ is also defining. 
 Let $l: \mathcal{A} \to \mathbf{X}$ be related to
the integrability of $f$.
%By the integrability of $f$ we have $l: \mathcal{A} \to \mathbf{X}$. 
Since $l$ does not depend on the defining sequence, we obtain that $l$ 
is also the limit of  the $f_n^+$'s integrals, which are non negative. Hence,  
by Axiom \ref{convergenze}.b), we obtain
that $l \geq_{X} 0$. These properties imply also that the integral is monotone (that is for every $f,g$ integrable  functions with $ f \leq_{X_1} g$ 
 it is $\displaystyle{
\int_G f d\mu \leq_{X} \int_G g d\mu}$ \, ).
\end{itemize}
\end{remark}

The next step is to prove the $\mu$-absolute continuity of the integral.
The following result extends \cite[Theorem 3.14]{BC2009}.
\begin{theorem}\label{ac}
Let $f : G \to {\mathbf{X}}_1$ be an integrable function. Then, the integral  $\displaystyle{\int_{(\cdot)} f(g) \, d\mu(g)}$ is $\mu$-absolutely continuous, that is
\begin{itemize}
\item[{\rm \ref{ac}.1)}] 
$\displaystyle{ \ell \left( \left(\int_{A_n}\,\vert f(g) \vert  \, d\mu(g)\right)_n \right)=0}$\,\,  whenever  \,\, $\ell_2((\mu( A_n))_n)=0$;
\item[{\rm \ref{ac}.2)}] 
there exists an increasing sequence  $(B_m)_m$ in ${\mathcal A}$, such that 
%$\displaystyle{\bigcup_{m\in \mathbb{N}} B_m=G}$,
$\mu(B_m)\in {\mathbf{X}}_2 $ for all $m \in \mathbb{N}$, and
%a strictly increasing sequence $(n_m)_m $ in $\mathbb{N}$, with
\begin{eqnarray*}\label{acac2}
	\ell\Bigl(\Bigl(\int_{G \setminus B_m} \vert f(g) \vert  \, d\mu(g)  \Bigr)_m \Bigr)=0.
\end{eqnarray*}
\end{itemize}
\end{theorem}
\begin{proof}
Without loss of generality, we can suppose that $f \geq_{X_1} 0$. Let $(f_n)_n$ be a defining sequence for $f$.
We begin with proving \ref{ac}.1). Let $(A_n)_n$ be a sequence of elements 
of ${\mathcal A}$ with  $\ell_2((\mu(A_n))_n)=0$.  For each $n \in \mathbb{N}$, it is
\begin{eqnarray}\label{l0}
	 0 &\leq_{X}&  
%I_n^0=
 \nonumber \int_{A_n} \, f(g)\, d\mu(g) \leq_{X} \Bigl\vert  \int_{A_n} f(g) \,
	d\mu(g) -\int_{A_n} f_n(g) \,d\mu(g) \Bigr\vert  +  \Bigl\vert  \int_{A_n} f_n(g) \, d\mu(g)\Bigr\vert  \leq_{X} \\   
	&\leq_{X}& \bigvee_{A \in {\mathcal A}} \, \left\vert  \int_A f(g) \, d\mu(g) -\int_A f_n(g) \,d\mu(g) \right\vert  + \int_{A_n} \vert f_n(g)\vert  \, d\mu(g).
%=I_n^1+I_n^2.
\end{eqnarray}
By (\ref{l}), we get 
\begin{eqnarray}\label{l1} 
\ell \Bigl( \Bigl(\bigvee_{A \in {\mathcal A}} \, \Bigl\vert  \int_A f(g) \, d\mu(g) -\int_A f_n(g) \,d\mu(g) \Bigr\vert \Bigr)_n\Bigr)=0.
\end{eqnarray}
From condition \ref{eac}.1) of $\mu$-equiabsolute continuity of  the integrals of the $f_n$'s,   we get 
\begin{eqnarray}\label{l2}
	\ell\Bigl(\Bigl(\int_{A_n} \vert f_n(g)\vert  \, d\mu(g)\Bigr)_n\Bigr)=0.
\end{eqnarray}
From (\ref{l0}), (\ref{l1}), (\ref{l2}) and 
 Axioms \ref{convergenze} we obtain
%\begin{eqnarray}\label{lfin}
$\displaystyle{\ell \Bigl(\Bigl( \int_{A_n} \, f(g)\, d\mu(g) 
\Bigr)_n\Bigr)=0}$, \, that is  condition \ref{ac}.1).
\\
Now we turn to \ref{ac}.2). Let $(f_n)_n$ be a defining sequence for $f$ and $(B_m)_m$ be a sequence from ${\mathcal A}$, according to
condition \ref{eac}.2) of the $\mu$-equiabsolute continuity of  the integrals of the $f_n$'s.  For every $n$, $m \in \mathbb{N}$, it is
\begin{eqnarray*}\label{l3prepre} 
	0 &\leq_{X} &  \nonumber \int_{G \setminus B_m}f(g)\, d\mu(g) \leq_{X} \Bigl\vert  \int_{G \setminus B_m} f(g) \,  d\mu(g) -\int_{G \setminus B_m} f_n(g) \,d\mu(g) \Bigr\vert + \\ 
	&+&  \Bigl\vert  \int_{G \setminus B_m}f_n(g) \, d\mu(g)\Bigr\vert  \leq_{X}  \bigvee_{A \in {\mathcal A}} \, \left\vert  \int_Af(g) \, d\mu(g) -\int_A f_n(g) \,d\mu(g) \right\vert  + 
	\nonumber
	\\ &+& 
	\int_{G \setminus B_m} \vert f_n(g)\vert  \, d\mu(g). %=  J_1 + J_2. 
\end{eqnarray*}
By the definition of integrability (\ref{l}), it is
\[\ell \Bigl( \Bigl( 
\bigvee_{A \in {\mathcal A}} \, \Bigl\vert  \int_Af(g) \, d\mu(g) -\int_A f_n(g) \,d\mu(g) \Bigr\vert  \Bigr)_n\Bigr)=0.\]
Thus, taking in 
%(\ref{l3prepre}) 
the previous inequality the limit $\ell$ and the limit superior $\overline{\ell}$
as $n$ tends to $+\infty$, by virtue of Axioms \ref{convergenze} and \ref{limsuppresentation} we obtain
\begin{eqnarray*}\label{I0I1I2}
0&\leq_{X}&  \ell\Bigl(\Bigl( \int_{G \setminus B_m}f(g)\, d\mu(g)\Bigr)_n\Bigr) \leq_{X} \ell\Bigl(\Bigl( \bigvee_{A \in {\mathcal A}} \, \Bigl\vert  \int_Af(g) \, d\mu(g) -\int_A f_n(g) \,d\mu(g) \Bigr\vert  \Bigr)_n\Bigr) + \\ &+&
\overline{\ell}\Bigl(\Bigl(
\int_{G \setminus B_m} \vert f_n(g)\vert  \, d\mu(g)\Bigr)_n\Bigr)=
	\overline{\ell}
\Bigl(\Bigl(\int_{G \setminus B_m} \vert f_n(g)\vert  \, d\mu(g)\Bigr)_n\Bigr) \quad \text{for    each    }m \in \mathbb{N}.
\end{eqnarray*}
 Taking the
%in (\ref{I0I1I2}) 
limit $\ell$ as $m$ tends to $+\infty$,  by condition \ref{eac}.2)  we have
\begin{eqnarray*}
0 \leq_{X} \ell\Bigl(\Bigl( \int_{G \setminus B_m}f(g)\, d\mu(g) \Bigr)_m\Bigr) \leq_{X} \ell \Bigl( \Bigl( \overline{\ell} 
\Bigl(\Bigl(\int_{G \setminus B_m} \vert f_n(g)\vert  \, d\mu(g)
\Bigr)_n \Bigr)\Bigr)_m\Bigr)=0,
\end{eqnarray*}
that is 
$\displaystyle{\ell\Bigl(\Bigl( \int_{G \setminus B_m}f(g)\, d\mu(g) \Bigr)_m\Bigr)=0}$. This ends the proof.
\end{proof}
 We now give two convergence results   for integrable functions.
\begin{theorem}\label{vitali0} {\rm (Vitali)}
Let $(f_n)_n$ be a sequence of integrable functions 
in ${\mathbf{X}}_1^G$, convergent in measure to $0$ 
on every set $A \in {\mathcal A}$ with $\mu(A) \in {\mathbf{X}}_2$ and with $\mu$-equiabsolutely continuous integrals. Then $(f_n)_n$ converges in $L^1$ to $0$.
\end{theorem} 
\begin{proof}
The result follows by arguing analogously as in the proof of Theorem \ref{tredieci}.
\end{proof}
\begin{corollary}\label{lebesguedominated} {\rm(Lebesgue)}
Let $(f_n)_n$ be a sequence of integrable functions 
in ${\mathbf{X}}_1^G$, conver\-gent in measure to $0$ on any set $A \in {\mathcal A}$ with $\mu(A) \in{\mathbf{X}}_2$, and let there be
an integrable map $h \in {\mathbf{X}}_1^G$ with 
%\begin{eqnarray}\label{domination}
$\vert f_n(g)\vert \leq_{X_1} h(g)$ for all $ n \in \mathbb{N}$ and $g \in G$. 
%\end{eqnarray}
Then $(f_n)_n$ converges in $L^1$ to $0$.
\end{corollary}
 \begin{proof} 
Since $h$ is integrable, then its integral is absolutely continuous.
From this and 
%(\ref{domination}) 
the boundedness of the $f_n$'s
we deduce that the  integrals of the $f_n$'s are 
$\mu$-equiabsolute continuous.
Therefore, Corollary \ref{lebesguedominated} is a consequence of Theorem \ref{vitali0}. 
\end{proof}
We now prove that, in the classical case, the 
 above defined integral
% defined in \ref{integrabilita} and \ref{integrabilitaa} 
and the Lebesgue's one are equivalent
(for a related overview see, e.g., \cite{HALMOS}).\\
A measure $\mu:{\mathcal A}\to
\overline{\mathbb{R}}^{+}_0$  is said to be \emph{regular} iff $\mu(C) < + \infty$
for each compact set $C \subset G$, and  for every $\varepsilon \in \mathbb{R}^+ $
and $E \in {\mathcal A}$ there are a compact set $C \subset E$ and an open  set $U \supset E$ with $\mu (U \setminus C) $
 $\leq$ $\varepsilon$. 
Note that every regular finitely additive measure is  
$\sigma$-additive.

\begin{proposition}\label{lebesgue}
Let $G=(G,d)$ be a metric space,  ${\mathbf{X}}={\mathbf{X}}_1 ={\mathbf{X}}_2= {\mathbb{R}}$  be endowed with 
the usual convergence, ${\mathcal A}$ be the $\sigma$-algebra of  all measurable subsets of $G$, 
and $\mu:{\mathcal A}\to \mathbb{R}$ be a  $\sigma$-finite regular measure.
Then, a function  $f \in {\mathbb{R}}^G$ is integrable 
if and only if it is Lebesgue integrable, and in this case we have
\begin{eqnarray*}\label{equivalence}
\int_A 
f(g) \, d\mu(g)=(L) \int_A f(g) \, d\mu(g) \quad \text{ for   all   } \, A \in {\mathcal A},
\end{eqnarray*}
where $\displaystyle{(L) \int_{(\cdot)}}$ denotes the Lebesgue integral.
\end{proposition}

\begin{proof}
Let $f:G \to{\mathbb{R}}$ be Lebesgue integrable on $G$.
 Without loss of generality, we can assume $f \geq 0$. 
Then there are an increasing sequence  $(f_n)_n$ in ${\mathbb{R}}^G$, $\mu$-convergent  to $f$
on $G$, and hence $\mu$-convergent  to $f$ on every set $A \in {\mathcal A}$ with $\mu(A) < + \infty$.
Condition \ref{eac}.1) on the 
$\mu$-equiabsolute continuity  of the integrals of the $f_n$'s 
follows from the absolute continuity of the  Lebesgue integral, while condition \ref{eac}.2) is a consequence 
of the $\sigma$-additivity of regular measures and of
%$\sigma$-additivity of 
the Lebesgue integral. Thus, $(f_n)_n$ is a defining sequence for $f$.
From this and Proposition \ref{welldefined} it follows that every Lebesgue integrable function is integrable according to  
Definitions \ref{integrabilita} and \ref{integrabilitaa}.
\\
Now we prove the converse implication. Let $f \in {\mathbb{R}}^G$ be  integrable 
%according to \ref{integrabilita} and \ref{integrabilitaa}, 
and let
$(f_n)_n$ be an associated defining sequence of $f$. 
Since $f_n \in \mathscr{S}$ for each $n \in \mathbb{N}$, then the $f_n$'s are Lebesgue integrable.  
Without loss of generality, we may suppose that 
$0 \leq f_n(g) \leq f_{n+1}(g) \leq f(g)$ for every $g \in G$.
By the absolute continuity  of our integral, $(f-f_n)_n$ is a defining sequence for the 
identically null function. By Theorem \ref{vitali0},   $(f-f_n)_n$ converges in $L^1$ to $0$, and hence the sequence 
$\displaystyle{\Bigl(\int_G f_n(g) \, d\mu(g)\Bigr)_n}$ 
is bounded. Thus, by the classical monotone convergence theorem (see, e.g., 
\cite[Theorem 27.B]{HALMOS}), we obtain that $f$ is Lebesgue  integrable.
\\
The equality of the integrals
%relation (\ref{equivalence}) 
follows from the fact that  the values of both integrals 
do not depend on the chosen respective defining sequences.
\end{proof}

\subsection{Some properties of the integral}

We now give some properties of the integral 
and  present the concepts of uniform continuity and convexity in the vector lattice context. We give
some versions of the Jensen inequality, which will be useful to prove some modular convergence
theorems. The proofs of the results of this subsection  are given in the Appendix.
Let  $(\mathbf{X}_1^{\prime}, 
\mathbf{X}_1^{\prime\prime}, 
\mathbf{X}_1)$, $(\mathbf{X}_1,\mathbf{X}_2,  \mathbf{X})$ 
be two product triples.

\begin{proposition}\label{intproduct}
Let
$h:G \to \mathbf{X}_1^{\prime}$, 
$q:G \to \mathbf{X}_1^{\prime\prime} $ be bounded and
integrable on $G$ according to Definitions \rm \ref{integrabilita} and \ref{integrabilitaa}. \em 
Then, the function $h \cdot q:G \to 
\mathbf{X}_1$  is bounded and integrable on $G$ too.
\end{proposition}
\begin{corollary}\label{intproduct2}
Let 
$h:G \to \mathbf{X}_1^{\prime}$, and
$q:G \to \mathbf{X}_1^{\prime\prime} $.
Assume that
$h$ is integrable on $G$ and 
bounded on every set  $B \in {\mathcal A}$ with 
$\mu(B) \in {\mathbf{X}}_2$.
Moreover, suppose that $q$ is bounded and integrable on $G$, and   that  there is 
%a set 
$\overline{B} \in {\mathcal A}$ with  
 $\mu(\overline{B}) \in {\mathbf{X}}_2$   and 
such that $q$ vanishes   on  $G \setminus \overline{B}$. 
Then, the function  $h \cdot q:G \to 
 \mathbf{X}_1$ is bounded and integrable on $G$.
\end{corollary}

Now we give the concept of uniform continuity  for vector lattice-valued functions (see, e.g., \cite[Definition 3.9]{BCSVITALI}).
Let $G=(G,d)$ be a metric space. 
We say that 
\begin{itemize}
\item $f:G \to \mathbf{X}_1$ is $\emph{uniformly continuous on }$  $G$ iff there exists an element $u \in {\mathbf{X}}_1^+$ such that for
each 
$\varepsilon \in \mathbb{R}^+ $ there is $\delta \in \mathbb{R}^+ $ with $\vert f(g_1) -f(g_2) \vert \leq_{X_1} \varepsilon \, u$ whenever $g_1$, $g_2\in G$,
$d(g_1,g_2) \leq \delta$.
\item 
A function $\psi:\mathbf{X}_1\to \mathbf{X}_1$ is said to be  \emph{uniformly continuous on }  $\mathbf{X}_1$ iff
for every $u \in {\mathbf{X}}_1^+$ and $ \varepsilon \in  \mathbb{R}^+ $
there exist $w \in {\mathbf{X}}_1^+$ and $
 \delta \in  \mathbb{R}^+ $ with
$\vert \psi(x_1)-\psi(x_2)\vert \leq_{X_1} \, \varepsilon \, w$ whenever $\vert x_1-x_2 \vert \leq_{X_1} \delta \, u$.
\end{itemize}
Observe that, when $G={\mathbf{X}}_1=\mathbb{R}$ endowed with the usual topology, the two presented  concepts of uniform continuity coincide with the classical one.
Indeed, it is enough to consider $u=w=\mathbf{1}$,  the constant function which associates the real number $1$ to every $g \in G$.\\

Now we present some properties on uniformly continuous functions defined in a metric space $(G,d)$.

\begin{proposition}\label{compuc}
Let $f:G \to {\mathbf{X}}_1$, $\psi:\mathbf{X}_1\to \mathbf{X}_1$ be uniformly continuous 
on $G$ and $\mathbf{X}_1$, respectively. Then the composite function $\psi \circ f: G \to {\mathbf{X}}_1$ is uniformly continuous on $G$.
\end{proposition}

The next result extends the first part of  \cite[Proposition 3.12.4]{BCSVITALI} to  this setting.
%=========================

%=========================
\begin{proposition}\label{boundedness}
Let $\mu$ be such that $\mu(K) \in X_2$\,  
for every compact subset $K \in {\mathcal A}$. Let 
$C \in {\mathcal A}$ be a compact set, and $f: G \to \mathbf{X}_1$ 
be a uniformly continuous function on $G$, such that 
$f(g)=0$ whenever $g \in G \setminus C$. Then, $f$ is bounded and integrable on $G$. 
\end{proposition}
Now,   to deal with 
%in order to consider 
Orlicz-type spaces in the setting of vector lattice-valued modulars, we investigate some properties of convex functions.
\begin{definition}\label{def-conv}
\rm Let $\mathbf{X}$ be any Dedekind complete product algebra.
A function  $\varphi : \mathbf{X} \to \mathbf{X}$ is said to be \textit{convex}  iff for each $v \in \mathbf{X}$ there is an element $\beta_v \in \mathbf{X}$ with
%\mg{support}
\begin{eqnarray*}\label{support}
	\varphi(s) \geq_{X} \varphi(v) + \beta_v (s-v) \quad \text{ for  all  } s \in \mathbf{X}
\end{eqnarray*}
and such that
%\begin{eqnarray}\label{betabeta}
	the   set  $E_u :=\{ \beta_v : v \in [-u,u]   \}$    is   order  bounded    in    $ {\mathbf{X}}$     for   every   $u \in {\mathbf{X}}^+ \setminus \{0\}$.
%\end{eqnarray}
\end{definition}

\begin{remark}
Note that, when
$\mathbf{X} = \mathbb{R}$, $\varphi$ is convex if and only if
\[ \varphi(t) \leq \varphi(t_1) +  \dfrac{\varphi(t_2) - \varphi(t_1)}{t_2 - t_1} \, (t-t_1) 
\mbox{  for all } t, t_1, t_2 \in \mathbb{R} \, \, \mbox{ with  }\,  t_1 < t < t_2 \]
(see \cite[Definition 4.3]{BCSLp}), and the set $E_u$
is order bounded for all $u \in \mathbb{R}^+$,  since every convex function $\varphi : {\mathbb{R}} 
\to {\mathbb{R}}$ is Lipschitz on every bounded subinterval of the real line.
\end{remark}

\begin{example}\label{variousexamples} 
Now we give some examples of convex functions in the vector lattice setting.
\begin{itemize}
\item[\ref{variousexamples}.a)]
Define $\varphi: \mathbf{X} \to \mathbf{X}$ by	$\varphi(x)= x^2, \,\, x \in \mathbf{X}$,
which makes sense, since $\mathbf{X}$ is a product algebra. It is % Now we claim that 
\begin{eqnarray}\label{sussmayryksi}
	s^2 \geq_{X} v^2 + 2 \, v (s-v) \quad \text{for  all  } s,v  \in \mathbf{X}.
\end{eqnarray}
From this it follows that  $\varphi$ is convex, since the set $E_u =\{2 v : v \in [-u,u]\}$ is evidently order bounded for each $u \in \mathbf{X}^+$.
 
\item[\ref{variousexamples}.b)]
More generally observe that, at least in certain cases, it is possibile to construct convex $\mathbf{X}$-valued functions
by starting with convex real-valued  functions. For example, let 
$\widehat{\varphi}: \mathbb{R} \to \mathbb{R}$ be a function of class $C^1({\mathbb{R}})$, and
$\mathbf{X}={\mathcal C}_{\infty}(\Omega)$  be as in \cite[Theorem 2.1]{FILTER}.
For each $x\in \mathbf{X}$ and $\omega \in \Omega$, set  
\begin{eqnarray}\label{sussmayrlifting}
	\varphi(x)(\omega)= \widehat{\varphi}(x(\omega)). 
\end{eqnarray}
Note that $\widehat{\varphi} \circ x \in \mathbf{X}$ and $\widehat{\varphi}^{\prime} \circ x \in \mathbf{X}$
whenever $x \in \mathbf{X}$, and hence the function $\varphi$ in  (\ref{sussmayrlifting}) is well-defined. 
The convexity of  $\widehat{\varphi}$ implies that
\begin{eqnarray}\label{sussmayrkaksi}
	\widehat{\varphi}(s(\omega)) \geq \widehat{\varphi}(v(\omega)) +\widehat{\varphi}^{\prime} (v(\omega)) 
	\, (s(\omega) - v(\omega))
\end{eqnarray}
for every $\omega \in \Omega \setminus N$, where  $N \subset \Omega$ is a suitable nowhere dense
set (in general, depending on $s$ and $v$). Since, by the  Baire category theorem, the complement of a meager subset
of $\Omega$ is dense in $\Omega$, and since   $\widehat{\varphi}^{\prime} \circ v$ is continuous, 
the inequality in $(\ref{sussmayrkaksi})$ holds for every  $\omega \in \Omega$.
Setting $\beta_v (\omega)=\widehat{\varphi}^{\prime} (v(\omega))$, $\omega \in \Omega$, 
we obtain that $\beta_v \in \mathbf{X}$. Moreover observe that, if $v \in [-u,u]$, then by the 
monotonicity of $\widehat{\varphi}^{\prime}$ it follows that
%\begin{eqnarray*}
$	\widehat{\varphi}^{\prime} (-u(\omega)) \leq
	\widehat{\varphi}^{\prime} (v(\omega)) \leq \widehat{\varphi}^{\prime} (u(\omega) ) $.
%\end{eqnarray*} 
Thus, the set $E_u$  is bounded in ${\mathcal C}_{\infty}(\Omega)$.
From this and $(\ref{sussmayrkaksi})$ we deduce the convexity of $\varphi$.
\\
An analogous property holds when ${\mathbf{X}}= L^0(\Omega, \Sigma, \nu)$, where  $\nu: \Sigma \to \mathbb{R}^{+}_0$ is a 
$\sigma$-additive and $\sigma$-finite measure.
% is as in \ref{example}.
Indeed, in this case it is enough to argue as above, and 
 we get 
the inequality in (\ref{sussmayrkaksi})
for all $\omega \in \Omega$,  directly from the convexity of $\widehat{\varphi}$. 

\item[\ref{variousexamples}.c)] 
 Let 
 $\mathbf{X} \subset {\mathcal C}_{\infty}(\Omega)$ be 
as in \cite[Theorem 2.1]{FILTER}, and set 
\begin{eqnarray*}
	\varphi(x)=\vert x \vert^p,  \,\, \widehat{\varphi}(t)=\vert t \vert^p, \quad
x \in \mathbf{X}, 
%\, \, 
%p \in \mathbb{N}, \, \, p \geq 3, \quad   
	\,\, t \in \mathbb{R}, \, \, p \in \mathbb{N}, \, \, p \geq 3.
\end{eqnarray*} 
Since 
$\widehat{\varphi}^{\prime}(t)=p \, \vert t \vert^{p-2} \, t$, $t \in \mathbb{R}$, then $\varphi \circ x$ and 
$\varphi^{\prime} \circ x$ belong to $\mathbf{X}$ for all $x \in \mathbf{X}$. 
From the convexity of $\widehat{\varphi}$ on $\mathbb{R}$, we obtain
\begin{eqnarray}\label{xplattices}
	\vert s(\omega)\vert^p \geq \vert v(\omega)\vert^p +  p \,\vert v(\omega)\vert^{p-2} \, v(\omega) \, \, (s(\omega)-v(\omega)), 
	\quad \omega \in \Omega \setminus N,
\end{eqnarray} 
where $N \subset \Omega$ is a suitable nowhere dense set. Taking into account that, by the  
Baire category theorem, the complement of a meager subset of $\Omega$ is dense in $\Omega$, we get that the inequality in (\ref{xplattices}) holds for any $\omega \in \Omega$.
From (\ref{xplattices}) it follows that $\varphi$ satisfies the inequality in (\ref{support}), 
taking $\beta_v=p \,\vert v \vert^{p-2} \, v$, $\omega \in \Omega$. Moreover, since $\beta_v(\omega)$ is well-defined and 
%\begin{eqnarray*}
$\vert \beta_v(\omega)\vert \leq p \, \vert u(\omega)\vert^{p-1} \text{   for    each   }\omega \in \Omega,$
%\end{eqnarray*}
then the set 
%\begin{eqnarray}\label{betabetabeta}
$E_u= \{ \beta_v : v \in [-u,u]   \}$
%\end{eqnarray} 
is order bounded in ${\mathcal C}_{\infty}(\Omega)$. Since, by the Maeda-Ogasawara-Vulikh representation theorem  \cite[Theorem 2.1]{FILTER}, $\mathbf{X}$ can be viewed as 
a solid subspace in ${\mathcal C}_{\infty}(\Omega)$, it follows that  the set %in (\ref{betabetabeta}) 
$E_u$ is order bounded also in ${\mathbf{X}}$.
Thus, $\varphi$ is convex. 
\end{itemize}
%============================0
\end{example}
Now we prove some fundamental properties of convex functions,  extending \cite[Theorem 4.5]{BCSVITALI}.

\begin{proposition}\label{beginning}
Let $\varphi: \mathbf{X}_1 \to \mathbf{X}_1$ be convex and such that $\varphi(0)=0$. Then $\varphi$ satisfies the following property:
\begin{eqnarray}\label{mozart0}
	\text{if   } x \in {\mathbf{X}}_1 \text{  and  } \xi \in \mathbb{R}, \,\, 0 \leq \xi \leq 1, 
	\text{   then  } \varphi(\xi x )\leq_{X_1} \xi \, \varphi(x).
\end{eqnarray}

\end{proposition}

\begin{proposition}\label{locallipschitzconvex}
Let $\varphi:\mathbf{X}_1 \to \mathbf{X}_1$ be a convex function.
Then for every $u \in \mathbf{X}_1^+ \setminus \{0\}$ there is an element  $\beta^*_u \in \mathbf{X}_1^+ \setminus \{0\}$ such that
\begin{eqnarray}\label{convexlipschitz}
	\vert \varphi(x_1)-\varphi(x_2)\vert \leq_{X_1} \beta^*_u \, \vert x_1 - x_2\vert   \text{    for   all    }x_1, x_2 \in [-u,u], 
\end{eqnarray}
where $[-u,u]=\{v \in \mathbf{X}_1$: $-u \leq_{X_1} v \leq_{X_1} u \}$.
\end{proposition}

\begin{proposition}\label{compositionmozarteum}
Let $f:G \to \mathbf{X}_1$ be uniformly continuous on $G$ and  vanishing outside of a compact set $C \subset G$, $\mu$ be regular %on $X_2$ 
%\mg{\tiny qui tolto $X_2$}
and
$\varphi: \mathbf{X}_1 \to 
\mathbf{X}_1$ be convex and with $\varphi(0)=0$. Then the composite function
$\varphi \circ f: G \to \mathbf{X}_1$ is uniformly continuous on $G$, integrable on $G$, and vanishes outside $C$. 
\end{proposition}

Now we give the following versions of the Jensen inequality,  which extends  \cite[Theorem 4.6]{BCSLp}.
\begin{theorem}\label{jensen1} 
Let $C \subset G$ be a compact set,
$({\mathbf{X}},\mathbb{R}, {\mathbf{X}})$ 
 be 
a product triple,  $\mu:{\mathcal A} \to 
\overline{\mathbb{R}}^{+}_0$ be a $\sigma$-finite regular measure, $\varphi: {\mathbf{X}} \to {\mathbf{X}}$ be a convex function with 
$\varphi(0)=0$,  $f: G \to {\mathbf{X}}$ be uniformly continuous  on $G$ and such that $f(g)=0$ for all $g \in G \setminus C$.
Let $h: G \to \mathbb{R}_0^+$  be a Lebesgue integrable function  such that
%\begin{eqnarray*}\label{density}
$	\int_G h(g) \, d\mu(g)=1,$
%\end{eqnarray*}
and bounded on each subset $B \subset G$ with $\mu(B) < + \infty$. Then, we get 
\begin{eqnarray*}\label{jensenmozart}
	\varphi \Bigl (\int_G h(g) \, f(g) \, d\mu(g) \Bigr) \leq_{X} \int_G h(g) \, \varphi(f(g)) \, \, d\mu(g) .
\end{eqnarray*}
\end{theorem}

\begin{corollary}\label{jensencoroll}
Under the same hypotheses as in Theorem \rm \ref{jensen1}, \em  assume that 
\begin{eqnarray*}\label{schumann0}
	0  < \int_G \, h(g) \, d\mu(g) \leq 1.
\end{eqnarray*} 
Then, 
\begin{eqnarray*}
	\varphi \Bigl (\int_G h(g) \, f(g) \, d\mu(g) \Bigr) \leq_{X} \int_G h(g) \, \varphi(f(g)) \, \, d\mu(g) .
\end{eqnarray*}
\end{corollary}
\section{Vector lattice-valued modulars}\label{modularriesz}
 In this section we 
 deal with modulars
in the setting of vector lattices (for a related literature, see, e.g., \cite{BMV, BCSVITALI,
BCSLp, KOZLOWSKI} and the 
 references 
therein). Let $T$ be a linear subspace of ${\mathbf{X}}_1^G$, such that 
$\vert f \vert  \in T$ whenever $f \in T$, and such that, if $f \in T$ and $A \in {\mathcal A}$, then $f \cdot \chi_A \in T$.
%\end{itemize} 
We say that a functional $\rho:T \to \overline{\mathbf X}^{+}$ is a \textit{modular} on $T$ iff
\begin{itemize}
\item[($\rho_0$)] $\rho(0)=0$;
\item[($\rho_1$)] $\rho(-f)=\rho(f)$ for all $f \in T$;
\item[($\rho_2$)] $\rho(\alpha_1 f + \alpha_2 h) \leq_{X} \rho(f) + \rho(h)$ for 
any $f$, $h \in T$ and $\alpha_1$, $\alpha_2 \in 
\mathbb{R}_0^{+}$ with
$\alpha_1 + \alpha_2 =1$.
\end{itemize}
%We say that:
%\begin{itemize}
%\item[($\rho_m$)  ] 
A modular $\rho$ is said to be \textit{monotone} iff $\rho(f) \leq_{X} \rho(h)$ for each $f$, $h \in T$ with $\vert f \vert \leq_{X_1} \vert h \vert$.
%In this case it is not difficult to see that, 
 Observe that, 
if $f \in T$, then $\vert f \vert  \in T$ and $\rho(f) = \rho(\vert f \vert )$ (see, e.g., \cite{BMV, BCSVITALI}).

We say that a modular $\rho$ is \emph{convex} iff  $\rho(\alpha_1 f + \alpha_2 h) \leq_{X} \alpha_1 \, \rho(f) + \alpha_2 \, \rho(h)$ 
whenever $f$, $h \in T$, $\alpha_1$, $\alpha_2 \in 
\mathbb{R}_0^{+}$, $\alpha_1 + \alpha_2 =1$.
\vspace{3mm}

%================================================================
%______________________________________ inizio applicazioni_____________________________
%==================================================================
We now give some examples of modulars. First, we formulate the following condition on functions $\varphi: {\mathbf{X}}_1
\to {\mathbf{X}}_1$ which will be  useful for proving monotonicity of modulars:
 %\mg{xmonotonia}
\begin{eqnarray}\label{xmonotonia}
\varphi (x_1 \vee x_2) \leq_{X_1} \varphi (x_1) + \varphi(x_2) \mbox{   for each    } x_1, x_2 \in {\mathbf{X}}_1^{+}.
\end{eqnarray}
%\begin{example}\label{sussmayrmonotonicity}
Let $\varphi$, $\widehat{\varphi}$ and $\Omega$ be as in  (\ref{sussmayrlifting}), 
where $\widehat{\varphi}$ is not necessarily convex or differentiable.
Observe that, in this case, the definition of $\varphi$ makes sense  if and only if $\widehat{\varphi} \circ x \in {\mathbf{X}}_1$
whenever $x \in {\mathbf{X}}_1$. If $\widehat{\varphi} \in {\mathbb{R}}^{\mathbb{R}}$ is increasing, then for every
$x_1$, $x_2 \in {
\mathbf{X}}_1^+$ and $\omega \in \Omega \setminus N$, where $N$ is a suitable nowhere dense 
subset of $\Omega$, it is 
\begin{eqnarray}\label{sussmayrincreasing}
	\widehat{\varphi}(\max\{x_1(\omega),x_2(\omega)\}) \leq 
	\widehat{\varphi}(x_1(\omega))+\widehat{\varphi}(x_2(\omega)).
\end{eqnarray}
 Thanks to the Baire category theorem,  $\Omega \setminus N$ is dense in $\Omega$, and hence the
inequality in (\ref{sussmayrincreasing}) holds for all $\omega \in \Omega$.
Thus, from (\ref{sussmayrlifting}) and (\ref{sussmayrincreasing})  we deduce (\ref{xmonotonia}). Some other examples of
functions satisfying (\ref{xmonotonia}) can be found in  \cite{BCSVITALI}.
 Observe that, proceeding analogously as in \cite[Proposition 3.1]{BCSVITALI}, it is possible to see that, 
if $\varphi: {\mathbf{X}}_1 \to {\mathbf{X}}_1$ is increasing on $
{\mathbf{X}}_1^{+}$, $\varphi(0)=0$, $\varphi$ satisfies
%eliminati gli ultimi vecchi red
 (\ref{xmonotonia}), and 
\begin{eqnarray}\label{preorlicz}
	{ \mathcal L}^{\varphi}= \Bigl\{ f \in  {{\mathbf{X}}_1}^G: 
	\, \, \int_G \varphi(\vert f(g) \vert ) \, d\mu(g) \text{   exists   in  } {\mathbf{X}} \Bigr\} ,
\end{eqnarray}
then the operator $\rho^{\varphi}$ defined by
\begin{eqnarray}\label{orliczmodular}
	\rho^{\varphi}(f)=\int_G \varphi(\vert f(g) \vert ) \, d\mu(g), \quad f \in {\mathcal L}^{\varphi},
\end{eqnarray} 
is a monotone modular  and, if $\varphi$ is convex, then  $\rho^{\varphi}$ is convex on the set of the positive functions 
of ${\mathcal L}^{\varphi}$.  The set 
\begin{eqnarray}\label{orliczriesz} 
	\displaystyle{L^{\varphi}(G)= \{f \in {\mathbf{X}}_1^G: 
(o)\text{-}\lim_{\alpha \to 0^+} 
	\rho^{\varphi}(\alpha \, f) = \bigwedge_{\alpha \in \mathbb{R}^+}  \rho^{\varphi}(\alpha \, f) =0  \}}
\end{eqnarray}
is the \emph{Orlicz space} generated by $\varphi$  (here, 
$\displaystyle{(o)\text{-}\lim_{\alpha \to 0^+}  \rho^{\varphi}(\alpha \, f)=0}$
means that there exists an $(o)$-sequence $(\sigma_l)_l$ in $X$ such that for every $l \in \mathbb{N}$ there is 
$\overline{\alpha} \in \mathbb{R}^+$ with $\rho^{\varphi}(\alpha \, f)\leq_{X} \sigma_l$
whenever $0 <\alpha \leq \overline{\alpha}$).

Thanks to the properties of  modulars, 
it is not difficult  to check that $L^{\varphi}(G)$ is actually a linear space
and that, if $\alpha \in \mathbb{R}^+$ and $\alpha \, f \in L^{\varphi}(G)$, then $\beta \, f \in
L^{\varphi}(G)$ for each $\beta \in \mathbb{R}^+$, $\beta < \alpha$.
%\end{remark}

The subspace of $L^{\varphi}$ defined by
\begin{eqnarray}\label{Evarphi}
	E^{\varphi}(G)=\{ f \in L^{\varphi}(G): \rho^{\varphi}
	(\alpha \, f) \in \mathbf{X} \text{  for  every   }\alpha \in
	 \mathbb{R}^+  \} 
%eliminati gli ultimi vecchi red
\end{eqnarray}
is called the \emph{space of the finite elements of} $L^{\varphi}(G)$.

A sequence $(f_n)_n$ %of functions 
 in $L^{\varphi}(G)$ is \textit{modularly convergent} to $f \in L^{\varphi}(G)$ iff 
\begin{eqnarray}\label{modularconvergence0}
	\displaystyle{{\ell} (( \rho^{\varphi}(\alpha(f_n-f)))_n)=0} \quad \text{for  at  least  one  } \alpha \in \mathbb{R}^+ .
\end{eqnarray}
Note that, since $\rho^{\varphi}$ is a monotone modular and $\alpha \in  \mathbb{R}^+ $ satisfies
the condition in (\ref{modularconvergence0}), thanks to 
 Axioms \ref{convergenze} we get
\begin{eqnarray}\label{solidity}
	\displaystyle{{\ell} ((\rho^{\varphi}(\beta(f_n-f)))_n)=0} \quad \text{for  all }
	\beta  \in {\mathbb{R}}^+, \, \, \beta < \alpha.
\end{eqnarray}
 
\section{The structural assumptions on the operators}\label{structural}
To prove our results about modular convergence  with respect to the convergences 
introduced axiomatically in 
Axioms \ref{convergenze} 
in the vector lattice setting, we give 
some structural hypotheses.\\ 

 We begin with the next technical assumption.
\begin{itemize}
\item[\bf{H*)}] If $e \in {\mathbf{X}}_1^+\setminus \{0\}$,
$V[e]$ is as in Section \ref{due} and 
%eliminati gli ultimi vecchi red
$(x_n)_n$ is a sequence in $V[e]$ such that 
$\ell_{\mathbb{R}}((\|x_n\|)_n )=0$, then 
$\ell_1((\vert x_n \vert)_n )=0$.
\end{itemize}
\begin{remark} \rm Taking into account 
(\ref{prime}) and Axioms \ref{convergenze},
it is not difficult to see that condition 
{\bf H*)} is satisfied, for instance, when 
$\ell_{\mathbb{R}}$ is the usual (resp., filter, almost, Ces\`{a}ro)
convergence on $\mathbb{R}$, 
% ${\overline{\ell}}_{\mathbb{R}}$, 
and $\ell_1$ is the usual 
(resp., filter, almost, Ces\`{a}ro)
order (or $(r)$-) convergence on ${{\mathbf X}}_1$.
\end{remark}
Now, similarly as in \cite{bm1}, we give the following
\begin{assumptions}\label{intro} \rm
Let $G$ be a metric space, ${\mathcal A}$ be the $\sigma$-algebra of all Borel subsets of $G$,  $\mathbf{X}_2=
\mathbb{R}$,
and $\mu:{\mathcal A} 
\to \overline{\mathbb{R}}_0^{+}$ be $\sigma$-finite and regular. We denote by 
${\mathcal C}_c(G)$ the space of all uniformly continuous functions  $f \in \mathbf{X}_1^G$ with compact support on $G$.
\begin{itemize}
\item[\ref{intro}.a)] Let  ${\mathcal M}$
be the class of all measurable functions 
$L:G \times G \to \mathbb{R}_0^{+}$
 with respect to the product
$\sigma$-algebra, such that the sections $L(\cdot,t)$ and $L(s, \cdot)$ are 
%Lebesgue 
 integrable (with respect to $\mu$) and  
bounded on every set of finite measure $\mu$ for every $t$, $s \in G$, respectively.

\item[\ref{intro}.b)] Let $\Psi$ be the family of all functions $\psi:{\mathbf{X}}_1^{+} \to {\mathbf{X}}_1^{+}$ such that
\begin{itemize}
\item[{\ref{intro}.b.1)}] 
$\psi$ is uniformly continuous and increasing on ${\mathbf{X}}_1^{+}$, $\psi(0)=0$ and 
$\psi(v) \in {\mathbf{X}}_1^+\setminus \{0\}$ whenever $v \in {\mathbf{X}}_1^+\setminus \{0\}$. 
\end{itemize}
Let $\Xi= (\psi_n)_n \subset \Psi$ be a sequence of functions  such that: 
\begin{itemize}
\item[{\ref{intro}.b.2)}] 
$(\psi_n)_n$ is \emph{equicontinuous at} 0, that is for every $u \in {\mathbf{X}}_1^+\setminus \{0\}$ and $\varepsilon \in  \mathbb{R}^+$
there are $w \in {\mathbf{X}}_1^+\setminus \{0\}$ and  
 $\delta \in \mathbb{R}^+$ with $\psi_n(x) \leq_{X_1}   \varepsilon \, w$ whenever $x \leq_{X_1} \delta \, u$ and  
$n \in \mathbb{N}$;
\item[{\ref{intro}.b.3)}] 
for every $v \in {\mathbf{X}}_1^{+}$ the sequence $(\psi_n(v))_n$ is \emph{order equibounded}, that is there exists
$A_v \in {\mathbf{X}}_1^+\setminus \{0\}$ with $\psi_n(v) \leq_{X_1} A_v$ for all $n \in \mathbb{N}$.
\end{itemize}
Let ${\mathcal K}_{\Xi}$ be the class of all sequences of functions $K_n: G \times G \times {\mathbf{X}}_1\to  {\mathbf{X}}_1$, $n \in \mathbb{N}$, such that:
\begin{itemize}
\item[{\ref{intro}.b.4)}] 
$K_n(\cdot, t ,u)$  and $K_n(s, \cdot ,u)$ are integrable on $G$ 
 with respect to the  measure $\mu$ for any $u \in {\mathbf{X}}_1$, $t,s \in G$ and $n \in \mathbb{N}$; 
\item[{\ref{intro}.b.5)}]
$K_n(s,t,0) = 0 $ for each $n \in \mathbb{N}$ and $s$, $t \in G$;
\item[{\ref{intro}.b.6)}]  
there are two positive sequences $(L_n)_n \subset {\mathcal M}$ and $(\psi_n)_n \subset \Psi$, with
\begin{eqnarray*}\label{lipschitz}
	\vert K_n(s,t,u)-K_n(s,t,v) \vert \leq_{X_1} L_n(s,t) \, \psi_n(\vert u-v \vert)
\end{eqnarray*}
 for each $n \in \mathbb{N}$, $s$, $t \in G$ and $u$, $v \in {\mathbf{X}}_1$.
\end{itemize}
\item[\ref{intro}.c)] Let $\varphi: {\mathbf{X}}_1 \to {\mathbf{X}}_1$ be a function, convex on ${\mathbf{X}}_1$, 
increasing on ${\mathbf{X}}_1^{+}$ and such that
\begin{itemize}
\item[{\ref{intro}.c.1)}] 
$\varphi(0)=0$;
\item[{\ref{intro}.c.2)}]
$\varphi(x) \in {\mathbf{X}}_1^+\setminus \{0\}$ whenever $x \in {\mathbf{X}}_1^+\setminus \{0\}$.
\end{itemize}
\item[\ref{intro}.d)] Let $\mathbb{K}=(K_n)_n \in {\mathcal K}_{\Xi}$, and  $\textbf{T}=(T_n)_n$ be a sequence of Urysohn-type operators defined by
\begin{align}\label{operatori}
	(T_n f)(s)=\int_G K_n(s,t,f(t)) \, d\mu(t), \quad s \in G,
\end{align}
where $f \in $ Dom $\textbf{T}=$ $\displaystyle{\bigcap_{n=1}^{\infty}}$ Dom $T_n$,   and 
for every $n \in \mathbb{N}$, Dom $T_n$ is the set on which $T_n f$ is well-defined.
\end{itemize}
\end{assumptions} 
Now we extend the concept of singularity  given in \cite{bm1} to the 
setting of the convergences introduced axiomatically in \ref{convergenze}.

\begin{definition}\label{intro2}
 \rm 
A family $\mathbb{K} \in {\mathcal K}_{\Xi}$ is said to be \emph{singular} iff
there are: two sequences $(L_n)_n$ and $(\psi_n)_n$, satisfying 
{\ref{intro}.b.6)},
%(\ref{lipschitz}),
an infinite set $H \subset \mathbb{N}$, 
and an element $D^{(1)} \in {\mathbb{R}}^+$, such that
\begin{itemize}
\item[{\ref{intro2}.a.1)}]
\begin{eqnarray*}\label{D11}
	\displaystyle{\int_G L_n(s,t) \, d\mu(t) \leq D^{(1)}} 
	\text{   for  every   } s \in G  \text{   and   } n \in H,
\end{eqnarray*}
\item[{\ref{intro2}.a.2)}]
\begin{eqnarray*}\label{D12}
	\displaystyle{\int_G L_n(s,t) \, d\mu(s) \leq D^{(1)}}
	\text{   for  any   } t \in G  \text{   and   } n \in H,
\end{eqnarray*}
\item[{\ref{intro2}.a.3)}] 
${\overline{\ell}}_{\mathbb{R}}((x_n)_{n\in \mathbb{N}})= {\overline{\ell}}_{\mathbb{R}}((x_n)_{n\in H})$ for every sequence $(x_n)_n$ in $\mathbb{R}^+_0$,
%\mg{$\leftarrow$ limsup usuale} NO
\item[{\ref{intro2}.a.4)}] 
${\overline{\ell}}((x_n)_{n\in \mathbb{N}})= {\overline{\ell}}((x_n)_{n\in H})$ for any sequence $(x_n)_n$ in ${\mathbf{X}}_1^+$. 
\end{itemize}
A singular family ${\mathcal K}$ is  \emph{$(M)$-singular} (resp.,  \emph{$(U)$-singular}) iff
\begin{itemize}
\item[\ref{intro2}.b.1)]
\begin{eqnarray*}\label{prejensen}
	\int_G L_n(s,t) \, d\mu(t) >0 \quad \text{for   any  } s \in G \text{    and   } n \in  H;
\end{eqnarray*} 
\item[\ref{intro2}.b.2)]
for each $A\in {\mathcal A}$ with $\mu(A) < + \infty$  there is a sequence $(A_n)_n$ in ${\mathcal A}$ with \mbox{$  \ell_2((\mu(A_n))_n) =0$} and such that
%\mg{su IR}
\begin{eqnarray*}\label{see}
	\displaystyle{ {\ell}_{\mathbb{R}} \Bigl( 
\sup_{s \in A \setminus A_n} \int_{G \setminus B(s,\delta)} L_n(s,t) \, d\mu(t)  }\Bigr)=0
\end{eqnarray*}
(resp., for every $A \in {\mathcal A}$ with $\mu(A) < + \infty$, it is 
\begin{eqnarray*}\label{check}
	\displaystyle{ {\ell}_{\mathbb{R}} \Bigl( 
\sup_{s \in A} \int_{G \setminus B(s,\delta)}  L_n(s,t) \, d\mu(t)  }\Bigr)=0   )
\end{eqnarray*}
for each $\delta \in  \mathbb{R}^+ $, where $B(s,\delta)=\{t \in G$: $d(s,t)\leq \delta\}$;
\item[\ref{intro2}.b.3)]  
for any $A \in {\mathcal A}$ with $\mu(A) < +\infty$ there are a sequence  $(D_n)_n$ in ${\mathcal A}$ with
\mbox{$  \ell_2((\mu(D_n))_n)=0$},  
an element $ z \in {\mathbf{X}}^+ \setminus \{0\}$ and an 
$(o)$-sequence $(\varepsilon_n)_n$ in $
\mathbb{R}^+ $, such that
\begin{eqnarray*}\label{Thetan}
\displaystyle{  
%\Theta_n(s)=
\bigvee_{u \in \mathbf{X}_1 \setminus \{0\}
} \Bigl( \bigvee_{s \in A \setminus   D_n} 
\Bigl\vert \int_{G} K_n(s,t,u) \, d\mu(t) - u
\Bigr\vert  \Bigr)   \leq_X \varepsilon_n \, z}
\end{eqnarray*}
for all $n \in H$
(resp.,  
for every $A \in {\mathcal A}$ with $\mu(A) < +\infty$
there exist $ z \in {\mathbf{X}}^+ \setminus \{0\}$
and an $(o)$-sequence $(\varepsilon_n)_n$ in ${\mathbb{R}}^+$, with
\begin{eqnarray*}\label{Thetanbis}
\displaystyle{  
%\Theta_n(s)=
\bigvee_{u \in \mathbf{X}_1 \setminus \{0\}
} \Bigl( \bigvee_{s \in A}
\Bigl\vert \int_G K_n(s,t,u) \, d\mu(t) - u
\Bigr\vert  \Bigr)  \leq_X \varepsilon_n \, z}
\end{eqnarray*}
for each $n \in H$).

\end{itemize}
\end{definition}
\begin{remark}\label{crucial}
\phantom{a}
\begin{itemize}
\item[ \ref{crucial}.a)] 
Note that, if $H=\mathbb{N}$,  condition \ref{intro2}.a.1) (resp.,
 \ref{intro2}.a.2) ) is equivalent to the order equiboundedness of the integrals
\[ \int_G L_n(s,t) \, d\mu(t) \, \text{ (resp.,  }\, \int_G L_n(s,t) \, d\mu(s)\, ), \, n \in \mathbb{N}, \, s \in G \, \text{ (resp., } t \in G). \]
\item[\ref{crucial}.b)]
 As an example, let ${\mathcal F}$ be any fixed free filter on ${\mathbb{N}}$.
A sequence  $(x_n)_n$ in $\mathbb{R}$ is said to be  (\emph{order}) \emph{${\mathcal F}$-bounded}
iff there are $M_0 \in {\mathbb{R}}$ and $H \in {\mathcal F}$ with
\begin{eqnarray}\label{fbdd}
 \vert x_n \vert \leq M_0 \text{    for  all    }n \in H
\end{eqnarray}
(see, e.g., \cite{BDMEDITERRANEAN}).
It is not difficult to see that, if the sequences 
\[ \sup_{s \in G} \, \int_G L_n(s,t) \, d\mu(t), \, \, \,
\sup_{t \in G} \, \int_G L_n(s,t) \, d\mu(s), \, \, \, n \in \mathbb{N},\]
are ${\mathcal F}$-bounded and $H$ is as in (\ref{fbdd}), then  $H$ satisfies the conditions in  \ref{intro2}.a.j),  j=1,2, 
with respect to the ${\mathcal F}$-limit superior, Now we
prove \ref{intro2}.a.3). To this aim,
it is enough to see that, for any sequence $(x_n)_n$ in 
${\mathbb{R}}$, 
$
%\begin{eqnarray}\label{FH}
\displaystyle{
\inf_{F \in \mathcal{F}} (\sup_{n \in F} x_n)=
\inf_{F \in \mathcal{F}} (\sup_{n \in F\cap H} x_n)
}.
$
%\end{eqnarray}
Indeed, since $F \cap H \in {\mathcal F}$ whenever
$F$, $H \in {\mathcal F}$,
the inequality $\leq$ 
%in (\ref{FH}) 
follows from
the properties of the infimum. Moreover, as
$\displaystyle{\sup_{n \in F} x_n \geq
\sup_{n \in F\cap H} x_n}$ for each $F \in {\mathcal F}$,
taking the infimum as $F$ varies in ${\mathcal F}$ we obtain 
the converse inequality. 
\\
Analogously as above, it is possible to check that condition \ref{intro2}.a.4) is fulfilled  
if we endow ${\mathbf{X}}_1$ with
$(o{\mathcal F})$-convergence. Moreover, when
we deal with usual, almost and Ces\`{a}ro (order)
convergence, we will take $H=\mathbb{N}$ in the 
definition of   $(M)$- and $(U)$-singularity.
\end{itemize}
\end{remark}

\section{The main results}\label{argument}
In this section we present general modular  con\-ver\-gen\-ce results for the involved operators,
in connection  with  
the convergences introduced  
 in Axioms \ref{convergenze}
%{\mg{\tiny senno' 1 da solo cosa e'?}}
extending \cite[Theorem 3]{bm1}.
We begin with the following result on equiabsolute continuity.

\begin{theorem}\label{principale}
Under Assumptions \rm \ref{intro}, \em assume that ${\mathbf{X}}_1={\mathbf{X}}$, 
$\mathbb{K}$ is singular and that the following condition
is satisfied:
\begin{itemize}\label{starproperty}
\item[{\rm \ref{principale}.1)}] 
for every compact set  $C \subset G$ there exists an increasing sequence $(B_m)_m$  of %compact 
subsets of $G$ 
 having finite measure $\mu$ and with
%{\mg{\tiny la compattezza serve? O serve solo di misura finita?
%Ad ogni modo noi abbiamo supposto che $\mu$ 
%sia regolare, che secondo alcuni testi implica
%automaticamente che sia finita su ogni compatto...
%Ma riguardando la dim., cosa serve? Che la successione 
%$(B_m)_m$ invada $G$? Ricontrollare a voce. Tra l'altro
%la finitezza della misura sembrerebbe essere più consone 
%allo spirito del lavoro del premio award} }
\begin{eqnarray}\label{propertystar} 
	\ell_{\mathbb{R}} \Bigl( \Bigl( \bar{\ell}_{\mathbb{R}}  \Bigl( \Bigl( \sup_{t \in C} \int_{G \setminus B_m}  L_n(s,t) \, d\mu(s) \Bigr)_n \Bigr) \Bigr)_m \Bigr)=0.
\end{eqnarray}
\end{itemize} 
Then there is a constant $\lambda \in \mathbb{R}^+$ such that, for every  $f \in {\mathcal C}_c(G)$, the sequence 
$\varphi(\lambda\vert   T_n f  \vert )$, $n \in H$,  has 
 $\mu$-equiabsolutely continuous integrals. 
\end{theorem}

\begin{proof}
Let $H$ and $D^{(1)}$ be as in \ref{intro2}.a.1) and \ref{intro2}.a.2), $\lambda \in \mathbb{R}^+$ be such that 
$\lambda \, D^{(1)} \leq 1$, $f \in {\mathcal C}_c(G)$, $C$ be a compact set such that $f(s)=0$ for all $s \in G \setminus C$, 
$\displaystyle{u^*=\bigvee_{s \in G} \vert f(s) \vert}$, and $(B_m)_m$ be a sequence from ${\mathcal A}$, according to  (\ref{propertystar}).

 Set $\Theta^*= \displaystyle{\varphi\Bigl(\bigvee_{n \in \mathbb{N}} \psi_n (u^*)\Bigr)}$.
Note that $\Theta^* \in \mathbf{X}$, thanks to the order equiboundedness of the sequence $((\psi_n(u^*))_n)$ (see \ref{intro}.b.3)).
Pick arbitrarily $s \in G$ and  $n\in H$. 
Taking into account the monotonicity of $\varphi$ and using Corollary \ref{jensencoroll}, where the roles of $f$ and $h$  are played by the functions 
%\begin{eqnarray*}\label{jensenschumann}
$\lambda \, D^{(1)} \, \psi_n(\vert f(\cdot)\vert )$, 
$\dfrac{L_n(s,\cdot)}{D^{(1)}},$
%\end{eqnarray*}
respectively, 
%(with $n$ and $s$ fixed), 
we have
\begin{eqnarray}\label{jensen}
& &\varphi \Bigl(\lambda \int_G \vert K_n(s,t,f(t)) \vert \, d\mu(t) \Bigr)  \leq_{X} \nonumber
\varphi \Bigl(\lambda \int_G L_n(s,t) \, \psi_n(\vert f(t) \vert) \,d\mu(t) \Bigr)  \leq_{X}
\\ &\leq_{X}&  \varphi \Bigl(\lambda\, D^{(1)} \int_C \dfrac{L_n(s,t)}{D^{(1)}} \, 
 \psi_n(\vert f(t) \vert) \, d\mu(t) \Bigr)
 \leq_{X}
\nonumber
\\& \leq_{X}& 
\frac{1}{D^{(1)}} \int_C  L_n(s,t) \, \varphi(\lambda \, D^{(1)}
\, \psi_n(\vert f(t) \vert) \,) \, \, d\mu(t)  \leq_{X} \\& \leq_{X}& 
\nonumber
\frac{1}{D^{(1)}} \int_C  L_n(s,t) \cdot \Bigl( \varphi \Bigl(\bigvee_{n \in \mathbb{N}}
\psi_n (u^*) \Bigr)  \Bigr) \, d\mu(t)  \leq_{X}
%\\&\leq& 
%\nonumber
\frac{1}{D^{(1)}}  \, \Bigl( \int_C  L_n(s,t) \, d\mu(t) \Bigr) \cdot \Theta^* .
\end{eqnarray} 
By applying the Fubini-Tonelli theorem to $L_n$ and integrating
with respect to $\mu(s)$ on \mbox{$G \setminus B_m$,} $m \in \mathbb{N}$,
we  get, for each $m \in \mathbb{N}$ and 
$n \in H$, 
\begin{eqnarray}\label{tonelli}
\int_{G \setminus B_m} \Bigl(\int_C L_n(s,t) \, d\mu(t) \Bigr) \, d\mu(s) =
\int_C \Bigl(\int_{G \setminus B_m} L_n(s,t) \, d\mu(s) \Bigr) \, d\mu(t),
\end{eqnarray} and hence
\begin{eqnarray}\label{thetatheta}
& & \int_{G \setminus B_m} \Bigl(\dfrac{1}{D^{(1)}} \nonumber
\int_C L_n(s,t) \, d\mu(t) \Bigr) \cdot \Theta^*\, \Bigr) d\mu(s) =
\\
 &=& \dfrac{1}{D^{(1)}}
\Bigl(\int_C \Bigl(\int_{G \setminus B_m} L_n(s,t) \, d\mu(s) \Bigr) \, d\mu(t) \Bigr) \cdot
\Theta^*   \leq_{X} \\ & \leq_{X}& \dfrac{1}{D^{(1)}} \nonumber
\Bigl(\int_C \Bigl( \sup_{t \in C}
\int_{G \setminus B_m} L_n(s,t) \, d\mu(s) \Bigr) \, d\mu(t) \Bigr)
\cdot \Theta^*  \leq_{X} \dfrac{1}{D^{(1)}} \,
\iota_{m,n} \, \mu(C)\, \cdot \Theta^* ,
\end{eqnarray} 
where
%\begin{eqnarray*}\label{iotamn}
$\displaystyle{
\iota_{m,n}= \sup_{t \in C}
\int_{G \setminus B_m} L_n(s,t) \, d\mu(s)}.$
From (\ref{operatori}), (\ref{jensen}), (\ref{tonelli}) and (\ref{thetatheta}) we obtain
\begin{eqnarray}\label{fin}
0 & \leq_{X}&\int_{G\setminus B_m} \varphi(\lambda \vert T_n(f)(s) \vert   )\,
d\mu(s)  \leq_{X} \\ & \leq_{X}&
\int_{G \setminus B_m}
\varphi \Bigl(\lambda \int_G \vert K_n(s,t,f(t)) \vert \, d\mu(t) \Bigr) 
\, d\mu(s) \nonumber  \leq_{X}
\dfrac{\iota_{m,n} \, \mu(C)}{D^{(1)}}  \, \Theta^* .
\end{eqnarray}
By \ref{principale}.1), we know that 
$\ell_{\mathbb{R}} ( (
\bar{\ell}_{\mathbb{R}} ( ( \iota_{m,n})_{n\in\mathbb{N}}
) )_{m\in\mathbb{N}} )=0$. From this and
\ref{intro2}.a.3) it follows that
$\ell_{\mathbb{R}} ( (
\bar{\ell}_{\mathbb{R}} ( ( \iota_{m,n})_{n\in H}
) )_{m\in\mathbb{N}} )=0$.
Thus, taking into account that $\mu(C) < + \infty$, 
 from (\ref{fin}), Axioms \ref{convergenze},  \ref{limsuppresentation} and  conditions \ref{ass}.7),
\ref{ass}.9) we deduce 
\begin{eqnarray*}\label{I2}
	\ell \Bigl( \Bigl( \bar{\ell} \Bigl( \Bigl(  \int_{G\setminus B_m} \varphi(\lambda \vert T_n(f)(s) \vert   ) \, d\mu(s) \Bigr)_{n \in H} 
\Bigr) \Bigr)_{m \in \mathbb{N}}\Bigr)=0.
\end{eqnarray*}
From this and \ref{intro2}.a.4) we obtain
\begin{eqnarray*}\label{I2bis}
	\ell \Bigl( \Bigl( \bar{\ell} \Bigl( \Bigl(  \int_{G\setminus B_m} \varphi(\lambda \vert T_n(f)(s) \vert   ) \, d\mu(s) \Bigr)_{n \in \mathbb{N}} 
\Bigr) \Bigr)_{m \in \mathbb{N}} \Bigr)=0.
\end{eqnarray*}
Thus, \ref{eac}.2) follows.

Now we prove \ref{eac}.1). Let  $(A_n)_n$ be any
sequence in ${\mathcal A}$, such that  $ \ell_2((\mu(A_n))_n)=0$. 
By arguing analogously as in (\ref{jensen}) and (\ref{thetatheta}), we  obtain
\begin{eqnarray*}\label{dfin}
	0 \leq_{X}
\int_{A_n} \varphi(\lambda \vert T_n(f)(s) \vert   )\, d\mu(s)
	  \leq_{X} \mu(A_n) \, \mu(C)  \, \, \Theta^* .
\end{eqnarray*}
From 
%(\ref{dfin})
the previous inequality  and Axioms \ref{convergenze}  we get 
%\begin{eqnarray*}\label{dfinfin}
\[\displaystyle{\ell\Bigl( \Bigl(
\int_{A_n} \varphi(\lambda \vert T_n(f)(s) \vert   )\,
d\mu(s) \Bigr)_n \Bigr)=0},\]
%\end{eqnarray*}
that is \ref{eac}.1). This ends the proof.
\end{proof}
A consequence of Theorem \ref{principale} is the following

\begin{corollary}\label{principalebis}
Under the same hypotheses as in  Theorem \rm \ref{principale}, \em
if $\mathbb{K}$ is singular, then there is a constant
$\beta \in \mathbb{R}^+$ such that, for every 
$f \in {\mathcal C}_c(G)$, the sequence 
$\varphi(\beta\vert T_n f-f \vert)$, $n \in H$, 
has  $\mu$-equiabsolutely continuous integrals. 
\end{corollary}
\begin{proof}
Let $\lambda$ be as in Theorem \ref{principale}, and choose
arbitrarily $f \in {\mathcal C}_c(G)$.
%By Remark \ref{pf} (a), $\vert f \vert $ is integrable too, and hence 
Note that $\lambda \, \vert f \vert  \in {\mathcal C}_c(G)$. By Proposition \ref{compositionmozarteum}, 
the function $s \mapsto \varphi(\lambda \, \vert f(s) \vert )$ is integrable on $G$. By Theorem \ref{ac}, the integral
\begin{eqnarray}\label{schumann3}
\int_{(\cdot)}  \varphi(\lambda \, \vert f(s) \vert ) \, d\mu(s)
\end{eqnarray}
is absolutely continuous. Moreover, since $\rho^{\varphi}$ is a modular, we get
\begin{eqnarray*}\label{modularartistic} \nonumber
\int_A  \varphi\Bigl(\dfrac{\lambda}{2} \, \vert T_n f (s) - f(s) \vert \Bigr) \, d\mu(s)  \leq_{X}
\int_A  \varphi(\lambda \, \vert T_n f(s)\vert  ) \, d\mu(s) + 
\int_A  \varphi(\lambda \, \vert f(s)\vert  ) \, d\mu(s)
\end{eqnarray*}
for every $A \in {\mathcal A}$. The assertion follows from %(\ref{modularartistic}),
the absolute continuity of the integral in (\ref{schumann3})
 and Theorem \ref{principale}, taking $\beta=%\dfrac{\lambda}{2}$.
\lambda/2$.
\end{proof}
The following result on convergence in measure (resp., uniform
 convergence)
extends \cite[Theorem 2]{bm1} to our setting.
\begin{theorem}\label{theoremeleven}
Under Assumptions \rm \ref{intro}, if ${\mathbf{X}}_1={\mathbf{X}}$ and $\mathbb{K} \in {\mathcal K}_{\Xi}$ is 
$(M)$-  \rm( \em resp., $(U)$-\rm) \em singular, then
there exists $\alpha \in \mathbb{R}^+$
such that, for every 
%uniformly continuous function
%$f \in {{\mathbf{X}}_1}^G  $ with compact support, 
$f \in {\mathcal C}_c(G)$,
the sequence 
\begin{eqnarray*}\label{varphi}
\varphi(\alpha \vert T_n f - f \vert ), \, \, n \in \mathbb{N},
\end{eqnarray*}
converges to $0$ in measure \rm(\em resp., uniformly\rm) \em
on every set $A \in {\mathcal A}$ with $\mu(A) < + \infty$.

Moreover, the sequence 
$(T_n f)_n$ converges to $f$
in measure \rm(\em resp., uniformly\rm) \em
on each set $A \in {\mathcal A}$ such that
$\mu(A) < + \infty$.
\end{theorem}
\begin{proof}
First, we prove 
the results concerning convergence in measure. 

Let $D^{(1)}$ be as in \ref{intro2}.1) and \ref{intro2}.2), 
 $\alpha \in  \mathbb{R}^+$ be such that  $2\, \alpha \leq 1$ and $4 \, \alpha \, D^{(1)} \leq 1$,
and $u \in {\mathbf{X}}^+ \setminus \{0\}$ be related to the uniform continuity of $f$.  
Fix arbitrarily
$\varepsilon \in \mathbb{R}^+ $.
Without loss of generality, we can suppose $\varepsilon \leq 1$. 
Since the sequence $(\psi_n)_n$ is equicontinuous at $0$
(see \ref{intro}.b.2)),
then there exist $w \in {\mathbf{X}}^+ \setminus \{0\}$ and $\sigma \in \mathbb{R}^+$
with \begin{eqnarray}\label{cric}
\psi_n(\vert x_1-x_2 \vert)  \leq_{X} \varepsilon \, w \text{   whenever   }
\vert x_1-x_2\vert  \leq_{X}
\sigma \, u \text{   and   } n\in \mathbb{N}.
\end{eqnarray} 
By the uniform continuity of $f$, in correspondence 
with $\sigma$ there is  
$\delta \in \mathbb{R}^+$ such that 
\begin{eqnarray}\label{croc}
\vert f(t)-f(s)\vert  \leq_{X} \sigma \,u
\text{   whenever   } t,s \in G, \, \, d(t,s)\leq \delta. 
\end{eqnarray}
Since $\varphi$ is convex, increasing on ${\mathbf{X}}^+ \setminus \{0\}$ and
$\varphi(0)=0$, 
from (\ref{cric}), (\ref{croc}) and
Proposition \ref{beginning} it follows that
\begin{eqnarray}\label{psinpsin}
\varphi(\psi_n(\vert f(t)-f(s) \vert))  \leq_{X} \varphi(\varepsilon \, w) 
 \leq_{X}
\varepsilon \, \varphi(w)
\end{eqnarray}
 whenever $ t,s \in G$, $d(t,s)\leq \delta$
and $n\in \mathbb{N}$.
Fix $A \in {\mathcal A}$ with $\mu(A) < + \infty$. 
Let $z \in {\mathbf{X}}^+ \setminus \{0\}$,
%eliminati gli ultimi vecchi red
$(D_n)_n$
be a sequence in $\mathcal{A}$ 
 and $(\varepsilon_n)_n$ be
an $(o)$-sequence in $\mathbb{R}^+$,
related to
\ref{intro2}.b.3). If $\varepsilon_1 >1$, we 
can replace $z$ with $\varepsilon_1 \, z$ 
and $(\varepsilon_n)_n$ with $(\varepsilon_n/\varepsilon_1)_n$
in \ref{intro2}.b.3, so that 
without loss of generality we can assume that $\varepsilon_1\leq 1$. Let $(A_n)_n$ be 
a sequence in $\mathcal{A}$
associated with \ref{intro2}.b.2),
%and the set $A$, 
and put $E_n= A_n \cup D_n$, $n \in \mathbb{N}$.  
As in (\ref{complementunion}), since 
$ \ell_2((\mu(A_n))_n)= \ell_2((\mu(D_n))_n)=0,$
we get
%\begin{eqnarray}\label{complementunionbis}
 $\ell_2((\mu(E_n))_n)=0.$ 
Now, set  
\begin{eqnarray}\label{units}
\displaystyle{v^*=\varphi(w) + \varphi(z) + \varphi \Bigl(4 \, \alpha
\, D^{(1)} \bigvee_{n \in \mathbb{N}}
\psi_n (2 \, u^*) \Bigr)}, 
\end{eqnarray}
and let $\delta $ be as in (\ref{croc}). 
By \ref{intro2}.b.2), we have
\begin{eqnarray}\label{deltaa}
\ell_{\mathbb{R}} \Bigl( \Bigl( 
\sup_{s \in   A \setminus E_n} 
\displaystyle{ 
\int_{G \setminus B(s,\delta)} L_n(s,t) \, d\mu(t) 
}\Bigr)_n \Bigr) =0.
\end{eqnarray}
%Taking into account 
By \ref{intro}.b.6),
(\ref{psinpsin}), (\ref{deltaa}),
%Proposition \ref{beginning} 
 thanks to the convexity of the modular $\rho^{\varphi}$,
and using Corollary \ref{jensencoroll}, where the roles of $f$ and $h$ 
are played by the functions \\
%\begin{eqnarray*}\label{jensenschumann2}
$4 \, \alpha\, D^{(1)} \, \psi_n(\vert f(\cdot)-f(s)\vert )$ and
$\dfrac{L_n(s,\cdot)}{D^{(1)}}$
%\end{eqnarray*}
respectively (with $n$ and $s$ fixed), for each 
$n \in H$  and
$s \in A \setminus E_n$ we get
\begin{eqnarray}\label{chopinstudies}
0 & \leq_{X}& \varphi(\alpha \vert (T_n f)(s)-f(s) \vert ) =\varphi\Bigr(\alpha \Bigl\vert 
\int_G K_n(s,t,f(t) )\, d\mu(t) -f(s)\Bigr\vert  \Bigr) \nonumber 
 \leq_{X} \\ & \leq_{X}&
\dfrac12 \, \varphi\Bigr(2 \, \alpha \Bigl\vert 
\int_G K_n(s,t,f(s)) \, d\mu(t) -f(s)\Bigr\vert  \Bigr) + 
\nonumber \\ &+&
\dfrac12 \, \varphi\Bigr(2 \, \alpha    \nonumber
\int_G \vert K_n(s,t,f(t))-K_n(s,t,f(s)) \vert  \, d\mu(t) \Bigr)  \leq_{X}\\  & \leq_{X}&
\dfrac12 \, \bigvee_{s \in A \setminus E_n} 
\,\varphi\Bigr(2 \, \alpha \Bigl\vert 
\int_G K_n(s,t,f(s)) \, d\mu(t) -f(s)\Bigr\vert  \Bigr) + 
\nonumber \\ &+&
\dfrac12 \, \varphi\Bigr(2 \, \alpha 
\int_G  L_n(s,t) \, \psi_n(\vert f(t)-f(s) \vert )\, d\mu(t) \Bigr)  \leq_{X} \nonumber \\ &  \leq_{X}&
\dfrac12 \, \varphi\Bigr(2 \, \alpha \bigvee_{s \in A \setminus E_n,
 f(s)\neq 0} 
\Bigl\vert  \int_G K_n(s,t,f(s)) \, d\mu(t) -f(s)\Bigr\vert  \Bigr) + 
\nonumber \\ &+&
\dfrac14 \, \varphi\Bigr(4 \, \alpha 
\int_{G \cap B(s,\delta)} L_n(s,t) \,\psi_n(\vert f(t)-f(s) \vert ) \, d\mu(t) \Bigr)+  \\
 &+& \dfrac14 \, \varphi\Bigr(4 \, \alpha \int_{G \setminus B(s,\delta)} 
\nonumber
L_n(s,t) \, \psi_n(\vert f(t)-f(s) \vert ) \, d\mu(t)\Bigr)  \leq_{X} \\
& \leq_{X}&
\dfrac12 \, \varphi\Bigr(2 \, \alpha 
\bigvee_{u \in {\mathbf{X}} \setminus \{0\}} \Bigl(
\bigvee_{s \in A \setminus E_n} 
\Bigl\vert  \int_G K_n(s,t,u) \, d\mu(t) 
-u\Bigr\vert  \Bigr) \Bigr)+ \nonumber
\\
\end{eqnarray}
\begin{eqnarray*}
  &+& \dfrac{1}{4 \,D^{(1)}} \int_{C \cap B(s, \delta)}
L_n(s,t) \, \varphi(4 \, \alpha \,D^{(1)} \nonumber
\psi_n(\vert f(t)-f(s) \vert ) \, d\mu(t)\Bigr) + \\
&+& \nonumber
 \dfrac{1}{4 \,D^{(1)}} \int_{C \setminus B(s, \delta)}
L_n(s,t) \, \varphi(4 \, \alpha \,D^{(1)}
\psi_n(\vert f(t)-f(s) \vert ) \, d\mu(t)\Bigr)  \leq_{X}\\  & \leq_{X}&
\dfrac12 \varphi( \varepsilon_n \, z) + \nonumber
\dfrac14 \varphi( \varepsilon \, w) +\frac14 \Bigl(
\int_{G\setminus B(s,\delta)} L_n(s,t) d\mu(t) \Bigr) \cdot \varphi
\Bigl( 4 \alpha D^{(1)} \bigvee_{n \in \mathbb{N}} \nonumber
\psi_n (2 u^*) \Bigr)  \leq_{X} \\  
 &\leq_{X}& 
\dfrac12 \, \varepsilon_n \, \varphi(z)+
\Bigl( \varepsilon + 
\sup_{s \in A \setminus E_n}
\int_{G\setminus B(s,\delta)} L_n(s,t) \, d\mu(t) \Bigr)
\cdot v^*  \nonumber \leq_X  \\  \nonumber
&\leq_{X}& 
\Bigl( \varepsilon + \varepsilon_n +
\sup_{s \in A \setminus E_n} 
\int_{G\setminus B(s,\delta)} L_n(s,t) \, d\mu(t) \Bigr)
\cdot v^* ,
\end{eqnarray*}
%%%%% 
 taking into account that $\varepsilon_1 \leq 1$
and by virtue of 
formula (\ref{mozart0}) in Proposition
\ref{beginning}.
Thus, all terms of (\ref{chopinstudies}) belong to the 
vector lattice $V[v^*]$, defined by 
\begin{eqnarray}\label{Vw}
V[v^*]=\{x \in {\mathbf{X}} : \text{    there   exists  } \sigma \in 
\mathbb{R}^+ \text{   with   } \vert  x \vert \leq \sigma \, v^* \}.
\end{eqnarray}
%As seen in Example \ref{Vv}, 
As seen in Section \ref{due}, 
$v^*$ is a strong order unit in
$V[v^*]$. Thus, taking in (\ref{chopinstudies}) the norm 
$\|\cdot\|_{v^*}$, defined analogously as in 
(\ref{latticenorm}),   from monotonicity 
and subadditivity of 
$\|\cdot\|_{v^*}$, 
\ref{intro2}.a.3), (\ref{triple}) and (\ref{deltaa}),
 and taking into account the monotonicity of 
$(\varepsilon_n)$ and \ref{limsupoperators}.f), we obtain 
\begin{eqnarray}\label{iunonorm}
0 &\leq& 
\nonumber \,  \bar{\ell}_{\mathbb{R}} \Bigl( \Bigl(
\Bigl\|\bigvee_{s \in A \setminus E_n}
\varphi(\alpha \vert (T_n f)(s)-f(s) \vert ) 
%\int_G \vert K_n(s,t,f(t))-K_n(s,t,f(s)) \vert \, d\mu(t) 
\Bigr\|_{v^*}\Bigr)_{n \in \mathbb{N}} \Bigr) =
\\&=& \nonumber \,  \bar{\ell}_{\mathbb{R}} \Bigl( \Bigl(
\Bigl\|\bigvee_{s \in A \setminus E_n}
\varphi(\alpha \vert (T_n f)(s)-f(s) \vert ) 
\Bigr\|_{v^*}\Bigr)_{n \in H} \Bigr)
\leq \\ &\leq& \bar{\ell}_{\mathbb{R}} \, \Bigl( \Bigl( \Bigl\|
\Bigl( \varepsilon +   \varepsilon_n \,+
\sup_{s \in A \setminus E_n} 
\int_{G\setminus B(s,\delta)} L_n(s,t) \, d\mu(t) \Bigr) 
\cdot v^*\Bigr\|_{v^*}\Bigr)_{n \in H} \Bigr) =
\\ &=& \bar{\ell}_{\mathbb{R}} \, \Bigl( \Bigl( \Bigl\|
\Bigl( \varepsilon +   \varepsilon_n \, +
\nonumber
\sup_{s \in A \setminus E_n}
\int_{G\setminus B(s,\delta)} L_n(s,t) \, d\mu(t) \Bigr) 
\cdot v^*\Bigr\|_{v^*}\Bigr)_{n \in \mathbb{N}} \Bigr) \leq
\\&\leq&  \varepsilon + \bar{\ell}_{\mathbb{R}}
\Bigl(  \Bigl( 
\sup_{s \in A \setminus E_n} 
\int_{G \setminus B(s,\delta)} L_n(s,t) \, d\mu(t) \Bigr)_{n
\in \mathbb{N}} \Bigr) +  \bar{\ell}_{\mathbb{R}}
((\varepsilon_n)_n)  \nonumber \leq \varepsilon . 
\end{eqnarray}
From (\ref{iunonorm}) and the arbitrariness of $\varepsilon 
\in\mathbb{R}^+$ it follows that 
\begin{eqnarray*}
 \bar{\ell}_{\mathbb{R}} \Bigl( \Bigl(
\Bigl\|\bigvee_{s \in A \setminus E_n}
\varphi(\alpha \vert (T_n f)(s)-f(s) \vert )
\Bigr\|_{v^*}\Bigr)_{n \in \mathbb{N}} \Bigr)=0,
\end{eqnarray*}
and hence, by Axiom \ref{limsuppresentation}.e), we obtain
\begin{eqnarray*}
\ell_{\mathbb{R}} \Bigl( \Bigl(
\Bigl\|\bigvee_{s \in A \setminus E_n}
\varphi(\alpha \vert (T_n f)(s)-f(s) \vert )
% \int_G \vert K_n(s,t,f(t))-K_n(s,t,f(s)) \vert \, d\mu(t) 
\Bigr\|_{v^*}\Bigr)_{n \in \mathbb{N}} \Bigr)=0.
\end{eqnarray*}
From this and condition {\bf H*) }
we deduce
\begin{eqnarray*}
\ell \Bigl( \Bigl(
\bigvee_{s \in A \setminus E_n}
\varphi(\alpha \vert (T_n f)(s)-f(s) \vert )
% \int_G\vert K_n(s,t,f(t))-K_n(s,t,f(s)) \vert \, d\mu(t) 
\Bigr)_{n \in \mathbb{N}} \Bigr)=0,
\end{eqnarray*}
and hence the sequence 
$
\varphi(\alpha \vert T_n f - f\vert ), \, \, n \in \mathbb{N},
$
converges 
in $\mu$-measure to $0$ on $A$.
The proof of the results about uniform convergence is analogous,
taking $A_n=D_n=\emptyset$.
To prove the last statement it is enough to argue analogously as above, 
taking $\alpha=1$ and $\varphi$ equal to the identity map.
\end{proof}
\vspace{3mm}
 
Now we define 
modular convergence 
in the context of 
abstract convergence in vector lattices.
\begin{definition}\label{modularconvergence}
\rm 
%(a) 
A sequence $(f_n)_n$ of functions in $L^{\varphi}(G)$
is \textit{${\ell}$-modularly convergent} to
$f \in L^{\varphi}(G)$ iff there exists a constant 
$\alpha \in \mathbb{R}^+$ such that
\begin{eqnarray*}\label{maintheorem1}
{\ell}((\rho^{\varphi}(\alpha(f_n-f))_n))=0.
\end{eqnarray*}
We are ready to present the main result on modular convergence 
in our setting, with respect to the convergences
introduced axiomatically in \ref{convergenze}.
\begin{theorem}\label{mainmodular1}
Under Assumptions \rm \ref{intro}, \em
if ${\mathbf{X}}_1={\mathbf{X}}$, 
$\mathbb{K}$ is $(M)$-singular 
and condition \rm \ref{principale}.1) \em holds, then
there is a constant 
$\alpha \in \mathbb{R}^+$ with
%\begin{eqnarray*}\label{maintheorem1m}
${\ell}((\rho^{\varphi}(\alpha(T_n f-f))_n))=0$
%\end{eqnarray*}
for every $f \in {\mathcal C}_c(G)$.
\end{theorem} 
\begin{proof}
Theorem \ref{mainmodular1} is a consequence of Corollary \ref{principalebis}, Theorem
\ref{theoremeleven}
and Theorem \ref{tredieci} applied to the sequence
$\varphi(\alpha \vert   T_n f - f \vert ),$ $n \in H$.
\end{proof}
\end{definition}

\subsection{Applications to Mellin kernel}\label{appl}
%{\mg{\tiny ????????? perche' non trattiamo piu' MGW e MPC
%e neppure li nominiamo. Quindi il titolo della subsection 
%e la frase di sotto vanno cambiati????}}  }
%We consider nonlinear
%Mellin operators and in particular moment, Mellin-Gauss-Weierstrass
%and Mellin-Poisson-Cauchy-type kernels 
%(see also \cite{BDMEDITERRANEAN}).
%\mg{nascoste le vecchie righe}
In this setting, we formulate the following structural assumptions:
\begin{assumptions} \label{assumptionsmellin}\rm
\begin{itemize}
\item[\ref{assumptionsmellin}.a)]
 Let $G=(\mathbb{R}^+,d_{\ln})$, where
%\begin{eqnarray*}\label{dist}
$d_{\ln} (t_1, t_2) =\vert \ln t_1 - \ln t_2\vert $, $t_1, t_2 \in \mathbb{R}^+ $,
%\end{eqnarray*}
let ${\mathbf{X}}_1={\mathbf{X}}$,
${\mathbf{X}}_2={\mathbb{R}}$,
and for any measurable set $S \subset \mathbb{R}^+$ put
$\displaystyle{\mu(S)= \int_S \frac{dt}{t}}.$
Let $\widetilde{\mathcal M}$ 
be the set of all sequences of non-negative functions
$\widetilde{L}_n$ defined on $\mathbb{R}^+$, 
 integrable 
with respect to $\mu$ and bounded on every subset 
$A \subset \mathbb{R}^+$ with $\mu(A) < +\infty$.

\item[\ref{assumptionsmellin}.b)] Let $\Xi= (\psi_n)_n \subset \Psi$ be as in \ref{intro} (c), and
denote by ${\widetilde{\mathcal K}}_{\Xi}$ the set of all sequences of functions
$\widetilde{K}_n: \mathbb{R}^+ \times \mathbb{R} \to 
{\mathbf{X}}$, $n \in \mathbb{N}$, such that:
\begin{itemize}
\item[\ref{assumptionsmellin}.b.i)]
 $\widetilde{K}_n(\cdot,u)$ is integrable for each $
u \in {\mathbf{X}}$ and $n \in \mathbb{N}$, and $\widetilde{K}_n(t,0) = 0 $ for all
$n \in \mathbb{N}$ and $t \in \mathbb{R}^+$;
\item[\ref{assumptionsmellin}.b.ii)]  
there are sequences $(\widetilde{L}_n)_n \subset \widetilde{\mathcal M}$ and
$(\psi_n)_n \subset \Psi$ with
\begin{align*}\label{mellinlipschitz}
\vert \widetilde{K}_n(t,u)-\widetilde{K}_n(t,v) \vert  \leq_{X} \widetilde{L}_n(t) \, \psi_n(\vert u-v\vert )
\end{align*}
for all $n \in \mathbb{N}$, $t \in \mathbb{R}^+$ and 
$u$, $v \in \mathbf{X}$.
\end{itemize}
\item[\ref{assumptionsmellin}.c)] Let $\widetilde{\mathbb{K}}=(\widetilde{K}_n)_n \in {\widetilde{\mathcal K}}_{\Xi}$, and consider
a sequence $\widetilde{\textbf{T}}=(\widetilde{T_n})_n$ of nonlinear 
Mellin-type operators defined by
\begin{eqnarray*}\label{mellinoperatori}
(\widetilde{T_n} f)(s)=\int_0^{+ \infty} \, \widetilde{K}_n\Bigl(\frac{t}{s},f(t)\Bigr)
\frac{dt}{t}, \quad s \in \mathcal{\mathbb{R}}^+,
\end{eqnarray*}
where $f \in $ Dom $\widetilde{\textbf{T}}=$ $\displaystyle{\bigcap_{n=1}^{\infty}}$
Dom $\widetilde{T_n}$, and Dom $\widetilde{T_n}$ is the set
on which $\widetilde{T_n} f$ is well-defined, for each 
$n \in \mathbb{N}$.
\end{itemize}
\end{assumptions}
It is not difficult to check that the function
$d_{\ln}$  is actually a distance. \\ \\
Now we give the concept of singularity in the setting of Mellin operators.
\begin{definition}\label{intro2mellin}
 \rm 
%(a) 
We say that $\widetilde{\mathbb{K}}= (\widetilde{K}_n)_n$ is 
\emph{singular} iff
there are  an infinite set 
$H \subset \mathbb{N}$ and
%an element 
$D^{(1)} \in {\mathbb{R}}^+$ with
\begin{itemize}
\item[{\ref{intro2mellin}.1)}]
%\begin{eqnarray*}\label{D*11}
$\displaystyle{
\int_0^{+ \infty} \widetilde{L}_n(t) \, \dfrac{dt}{t} \leq D^{(1)} }
\text{   for  every   } n \in H$,
%\end{eqnarray*}
and
\item[{\ref{intro2mellin}.2)}]
$\overline{\ell}((a_n)_{n\in \mathbb{N}})=
\overline{\ell}((a_n)_{n\in H})$
for each sequence $(a_n)_n$ in $\mathbb{R}$.
\end{itemize}
A singular family ${\mathcal K}$ is said to be
%\emph{$(M)$-singular} (resp., 
\emph{$(U)$-singular} iff:
\begin{itemize}
\item[{\ref{intro2mellin}.3)}]
%\begin{eqnarray*}\label{prejensenmellin}
$\displaystyle{
\int_0^{+ \infty} \widetilde{L}_n(t) \, \dfrac{dt}{t} >0} \quad \text{for  all } n \in \mathbb{N};$
%\end{eqnarray*}
\item[{\ref{intro2mellin}.4)}] 
for any 
$\delta \in \mathbb{R}^+$, $\delta > 1$, one has 
\begin{eqnarray*}\label{checkmellin}
\displaystyle{ {\ell}_{\mathbb{R}} \Bigl( \Bigl(
 \int_{\mathbb{R}^+ \setminus [1/\delta, \delta]}
 \widetilde{L}_n(t) \, \dfrac{dt}{t}
}\Bigr)_n \Bigr)=0;   
\end{eqnarray*} 

\item[{\ref{intro2mellin}.5)}] 
there are $z \in {\mathbf{X}}^+ \setminus \{0\}$
and an $(o)$-sequence $(\varepsilon_n)_n$ in ${\mathbb{R}}^+$, 
with
\begin{eqnarray*}\label{Thetanbis}
\bigvee_{u \in \mathbf{X}_1 \setminus \{0\}
} \Bigl\vert \int_0^{+\infty} \widetilde{K}_n(t,u) \, \dfrac{dt}{t} - u
\Bigr\vert \leq_X \varepsilon_n \, z
\end{eqnarray*}
whenever $n \in H$).
\end{itemize}
\end{definition}
\begin{remark}
\phantom{A}
 \begin{itemize}
\item[a)] 
Note that, if we set
%\begin{align*}
$\displaystyle{L_n(s,t)=\widetilde{L}_n\Bigl( \frac{t}{s}\Bigr)}$, 
$\displaystyle{K_n(s,t,u)=\widetilde{K}_n\Bigl( \frac{t}{s},u \Bigr)}$
%\end{align*}
for all $s$, $t \in \mathbb{R}^+$, $u \in {\mathbf{X}}$
and $n \in\mathbb{N}$, then it is not
difficult to see that, if $\widetilde{K}_n$, $\widetilde{L}_n$, $n \in \mathbb{N}$,
fulfil Assumptions \ref{assumptionsmellin}, then $K_n$, $L_n$, $n \in \mathbb{N}$,
satisfy Assumptions \ref{intro2}, and to check that
\begin{eqnarray*} 
&&\int_0^{+\infty} \widetilde{L}_n\Bigl(\frac{t}{s}\Bigr) \, \frac{dt}{t} =
\int_0^{+\infty} \widetilde{L}_n(z) \, \frac{s \, dz}{sz}=
\int_0^{+\infty} \widetilde{L}_n(z) \, \frac{dz}{z},\\
&& \int_0^{+\infty}
\widetilde{K}_n\Bigl(\frac{t}{s},u\Bigr) \, \frac{dt}{t} =
\int_0^{+\infty} \widetilde{K}_n(z,u)\frac{s \, dz}{sz} = \int_0^{+\infty} \widetilde{K}_n(z,u)\frac{dz}{z}
\end{eqnarray*}
for all $n \in \mathbb{N}$,
$s \in \mathbb{R}^+$ and $u \in \mathbf{X}$.
\item[b)]
%===========================
Analogously to 
 \cite[Example 1]{bm1} we have that , 
if $(\widetilde{L}_n)_n$
 satisfies (\ref{intro2mellin}.4), namely
\begin{eqnarray*}\label{checkaamu}
\displaystyle{ {\ell}_{\mathbb{R}} \Bigl( \Bigl(
 \int_{\mathbb{R}^+ \setminus [1/\delta, \delta]}
 \widetilde{L}_n(t) \, \dfrac{dt}{t}
}\Bigr)_n \Bigr)=0   
\end{eqnarray*} 
for each $\delta >1$, then for each compact 
subset $C \subset \mathbb{R}^+$ 
there exists a  set
$B \subset \mathbb{R}^+$ of finite $\mu$-measure such that  
\begin{eqnarray*}\label{aamu1}
\bar{\ell}_{\mathbb{R}}  \Bigl( \Bigl( \sup_{t \in C}\int_{{\mathbb{R}}^+ \setminus B} \widetilde{L}_n
\Bigl(\dfrac{s}{t}\Bigr) \, \frac{ds}{s} \Bigr)_n \Bigr)=0.
\end{eqnarray*}
In fact,
pick arbitrarily a compact set $C \subset \mathbb{R}^+$.
Without loss of generality, we may suppose
that $C=[1/M,M]$, where $M>1$.
% In correspondence with $M$, 
Let $\delta = 2\, M.$ For every 
$t \in C$ it is $\frac{\delta}{t}%\dfrac{2 \, M}{t} 
\geq 
%{2\, M} \, \frac{1}{M}=
2$, and $\frac{1}{\delta \, t}
%=\dfrac{1}{2 \, M \, t}
\leq%\frac{M}{2 \, M}=
\frac12$.
From this and the non-negativity of the $\widetilde{L}_n$'s
it follows that
\begin{eqnarray*}\label{mult} \nonumber
0 &\leq& \sup_{t \in C}
\int_{\mathbb{R}^+ \setminus [1/\delta, \delta]}  
\widetilde{L}_n
\Bigl(\dfrac{s}{t}\Bigr) \, \frac{ds}{s} \leq \sup_{t \in C}
\int_{0}^{1/\delta}
\widetilde{L}_n
\Bigl(\dfrac{s}{t}\Bigr) \, \frac{ds}{s}
+\sup_{t \in C}\int_{\delta}^{+\infty}\widetilde{L}_n
\Bigl(\dfrac{s}{t}\Bigr) \, \frac{ds}{s}
= \\&=& \sup_{t \in C}\int_{0}^{1/(\delta \, t)}\widetilde{L}_n (z)
\dfrac{dz}{z} + 
\sup_{t \in C}\int_{\delta/t}^{+\infty} \widetilde{L}_n (z)
\dfrac{dz}{z} \leq \\ &\leq& 
\int_{0}^{1/2}\widetilde{L}_n (z)
\dfrac{dz}{z} + \int_{2}^{+\infty} \widetilde{L}_n (z)
\dfrac{dz}{z} =J_n. \nonumber
\end{eqnarray*}
Since, by hypothesis, $\ell_{\mathbb{R}}((J_n)_n)=0$, 
then, taking $B=[1/\delta, \delta]$, the assertion follows from 
previous inequalities and Axioms \ref{convergenze}.
\end{itemize}
\end{remark}
The next result follows directly from Theorem \ref{mainmodular1},
adapting it to the setting of Mellin-type operators.
\begin{theorem}\label{mainmodular1mellin}
%\mbox{\rm \cite[Theorem 6.9]{BS2021}}
Under Assumptions \rm \ref{assumptionsmellin},
\em assume that $\widetilde{\mathbb{K}}$ is $(U)$-singular,
Then, for every $f \in {\mathcal C}_c(\mathbb{R}^+)$,
%Furthermore, if $\widetilde{\mathbb{K}}$ is $(U)$-singular,
the sequence $(\widetilde{T_n} f )_n$ 
is uniformly convergent to $f$  
on ${\mathbb{R}}^+$, and modularly convergent 
to $f$ with respect to the modular $\rho^{\varphi}$, where the constant $a$ in 
(\ref{modularconvergence0}) can be chosen
independently of $f$.
\end{theorem} 
 Indeed, observe
that condition \ref{starproperty}.1)
 follows
from (\ref{intro2mellin}.4).
%(\ref{starpropertystar}).
\vspace{3mm}
%
%===================================
%

%=======================

%\vspace{3mm}

In the linear case, a
particular type of Mellin-type kernels is the \textit{moment kernel}, defined by
\begin{eqnarray}\label{moment}
\widetilde{L}_n(t)=n \, t^n \chi_{(0,1)}(t), \quad t \in \mathbb{R}^+.
\end{eqnarray}
For any $n \in \mathbb{N}$, $t \in \mathbb{R}^+$ and 
$u \in \mathbf{X}$, set
\begin{eqnarray}\label{momentbis}
%\widetilde{L}_n(t)=M_n(t), \quad 
\widetilde{K}_n (t,u)=\widetilde{L}_n (t) \cdot u.
\end{eqnarray}
Observe that for any $\delta > 1$,
$n \in \mathbb{N}$ and $u \in \mathbf{X}$ we have
\begin{eqnarray}\label{presentazion}
\int_0^{+\infty} \widetilde{K}_n (t,u) \, \nonumber \frac{dt}{t} &=& \Bigl(
\int_0^{+\infty} \widetilde{L}_n(t) \, \nonumber \frac{dt}{t} \Bigr) u= n 
\Bigl( \int_0^1 t^{n-1} \, dt \Bigr) u \, dt=u; \\
\int_{\mathbb{R}^+ \setminus [1/\delta,\delta]} \widetilde{L}_n(t) \, \frac{dt}{t}&=&
n \int_0^{1/\delta} t^{n-1} \, dt= \Bigl( \frac{1}{\delta} \Bigr)^n.
\end{eqnarray} From (\ref{moment}), (\ref{momentbis}) and (\ref{presentazion}) it follows
that the conditions in \ref{intro2mellin} 
on $(U)$-singularity are fulfilled.\\

% Note that it is possible to check that these 
%conditions are satisfied also by other Mellin-type kernels,
%like for instance Mellin-Gauss-Weierstrass and 
%Mellin-Poisson-Cauchy-type kernels, which 
%we will treat in a forthcoming paper 
%(see, e.g., \cite{BDMEDITERRANEAN}).\\ 
Furthermore, by proceeding and arguing similarly as in 
\cite[Subsection 3.4]{BCSVITALI}, it is 
possible to see that our theory includes also stochastic integration,
for example It\^{o}-type integrals.
Indeed, we can treat the standard Brownian motion 
${\mathbf{B}}=(B_t)_{0 \leq t \leq T}$ defined on
a probability space $(\Omega, \Sigma, \nu)$ as
a function defined on $[0,T]$ and taking values in
${\mathbf{X}}_2=L^2=L^2(\Omega, \Sigma, \nu)$ endowed 
with order convergence, or order filter convergence
with respect to a fixed free filter ${\mathcal F}$ on ${\mathbb{N}}$.
Moreover, we can consider a stochastic process with
(uniformly) continuous trajectories
$f:G \to {\mathbf{X}}$, where $G=[0,T]$ is
endowed with the usual distance and 
${\mathbf{X}}=L^0=L^0(\Omega, \Sigma, \nu)$ is
equipped with order convergence,
and to define the integral of $f$ with respect to $\mathbf{B}$
as an element of ${\mathbf{X}}=L^0$,
according to \ref{integrabilita} and \ref{integrabilitaa}
(see, e.g., \cite{BCSVITALI} and the related references therein).
Note that, in this context, $({\mathbf{X}}, {\mathbf{X}}_2, {\mathbf{X}})$
is a product triple, and  Axioms \ref{convergenze}, 
%and
\ref{limsuppresentation} are satisfied (see Examples \ref{limsupoperators}).
%=============================
\section*{Appendix}\label{appendix}
\begin{proof}[Proof of Proposition \ref{intproduct}]
It is readily seen that $h \cdot q$ 
is bounded of $G$. 
Now we prove its integrability. Set
%\begin{eqnarray*}\label{boundednessh}
$\displaystyle{u= \bigvee_{g \in G} \vert h(g) \vert }$.
%\end{eqnarray*}
Note that $u \in {\mathbf{X}}_1^{\prime}$, 
since $h$ is bounded on $G$. Moreover, as $h$ and $q$ are  integrable on $G$,
there are two defining sequences $(h_n)_n$ and $(q_n)_n$ of simple functions for $h$ and $q$,
respectively. Pick arbitrarily $A \in {\mathcal A}$ with  $\mu(A) \in {\mathbf{X}}_2$. There are two sequences $(A_n)_n$ and $(D_n)_n$ in ${\mathcal A}$ with 
$\ell_2((\mu(A_n))_n)=$ $\ell_2((\mu(D_n))_n)=0$  and
\begin{eqnarray}\label{sussmayrmeasurepif} 
\ell_1^{\prime}\Bigl(\Bigl( \bigvee_{g \in A \setminus A_n} \,
 \vert h_n(g)-h(g) \vert  \Bigr)_n\Bigr) = 
	\ell_1^{\prime \prime}\Bigl(\Bigl(\bigvee_{g \in A \setminus D_n} \, \vert q_n(g)-q(g) \vert  \Bigr)_n\Bigr)=0.
\end{eqnarray}
Put $E_n=A_n \cup D_n$, $n \in \mathbb{N}$. As in  (\ref{complementunion}), we have
%\begin{eqnarray*}\label{complementunion1}
 $\ell_2( (\mu(E_n))_n)=0$. 
%\end{eqnarray*}
Since $q$ is bounded on $G$, without loss of generality  we can assume that
$\displaystyle{
v= \bigvee_{g \in G, n \in \mathbb{N}} \vert q_n(g) \vert  \in {\mathbf{X}}_1^{\prime \prime}}$.
Moreover, for all $g \in G$ and $n \in \mathbb{N}$ it is
\begin{eqnarray}\label{0crucial0}
	0 &\leq_{X_1}& \nonumber
	\vert h(g) \, q(g) -  h_n(g) \, q_n(g) \vert   \leq_{X_1} \vert h(g) \vert  \cdot \vert q_n(g) - q(g) \vert   + \vert h_n(g) -h(g) \vert  \cdot \vert q_n(g)\vert   \leq_{X_1} \\ &\leq_{X_1}&
	u \, \vert q_n(g) - q(g) \vert   + \vert h_n(g) -h(g)\vert  \, v .
\end{eqnarray} 
From (\ref{sussmayrmeasurepif}),  (\ref{0crucial0}), 
Axioms \ref{convergenze}
and conditions \ref{ass}.7), \ref{ass}.8) we deduce 
%ho tolto vecchi color red
\begin{eqnarray*}\label{productproduct}
	\displaystyle{\ell_1\Bigl(\Bigl( \bigvee_{g \in A \setminus E_n} \, \vert h_n(g) \, q_n(g)-h(g) \, q(g)\vert  \Bigr)_n\Bigr)=0}. 
\end{eqnarray*}
Thus, the sequence $(h_n \cdot q_n)_n$ converges in measure to $h \cdot q$.  \\
Now we claim the $\mu$-equiabsolute continuity of the integrals of the $h_n \cdot \, q_n$'s. For each  $n \in \mathbb{N}$ and $g \in G$, it is
%\begin{eqnarray*}
$0 \leq_{X_1} \vert  h_n(g) \, q_n(g) \vert \leq_{X_1} \vert h_n(g)\vert  \, v, $
%\end{eqnarray*} 
and hence
\begin{eqnarray*}\label{equiboundedness}
	0 \leq_{X} \int_A \vert h_n(g) \, q_n(g) \vert \, d\mu(g) \leq_{X}  
	\Bigl( \int_A \vert h_n(g)\vert  \, d\mu(g) \Bigr) \cdot v, \,\,\text{  for  all  } n \in \mathbb{N}
	\text{   and  } A \in {\mathcal A}.
\end{eqnarray*}
The claim follows from this inequality,
%(\ref{equiboundedness}), 
the $\mu$-equiabsolute continuity of the integrals of the
$h_n$'s,  Axioms \ref{convergenze}  and condition  \ref{ass}.7). 

Now we prove the existence of a map $l:{\mathcal A} \to \mathbf{X}$, satisfying formula 
(\ref{l}) in Definition \ref{integrabilita}. Without loss of generality, we can suppose that $h$, $q$, $h_n$ and $q_n$ are positive for every $n$, 
and that the sequences  $(h_n)_n$, $(q_n )_n$ are increasing. 
Arguing analogously as in (\ref{0crucial0}), for each  $n$, $m \in \mathbb{N}$ with $m \geq n$ we get 
\begin{eqnarray*}\label{einekleinenacht}
0 &\leq_{X} & \Bigl( \bigvee_{A \in {\mathcal A}}
\Bigl( \int_A h_m(g) \, q_m(g) \, d\mu(g) -\int_A h_n(g) \,  q_n(g) \, d\mu(g) \nonumber \Bigr) \Bigr) 
\leq_{X}  \\ &\leq_{X} &
u \Bigl(\int_G (q_m(g) - q_n(g)) \, d\mu(g) \Bigr) 
+ \Bigl( \int_G(h_m(g) -h_n(g)) \, d\mu(g) \Bigr) v \leq_{X}  \qquad
\\ &\leq_{X} & u \Bigl(\int_G (q(g) - q_n(g)) \, d\mu(g) \Bigr) 
+ \Bigl( \int_G (h(g) -h_n(g)) \, d\mu(g) \Bigr) v. \nonumber
\end{eqnarray*}
Since $q$ and $h$ are integrable on $G$
and their integrals are absolutely continuous, then
the sequences $(q_n - q)_n$ and $(h_n - h)_n$ 
are defining sequences for the identically zero function, 
and so both of them converge in $L^1$ to $0$, thanks to
Theorem \ref{vitali0}. From this last inequality,
%{ (\ref{einekleinenacht}), 
Axioms \ref{convergenze} and conditions
\ref{product}.1.7), \ref{ass}.8) it follows that 
\begin{eqnarray*}\label{conclusionproduct}
\ell\Bigl(\Bigl( \bigvee_{A \in {\mathcal A}}
\Bigl(\Bigl(   \bigvee_{m=1}^{\infty} 
\int_A h_m(g) \, q_m(g) \, d\mu(g)\Bigr) -\int_A h_n(g) \, 
q_n(g) \, d\mu(g) \Bigr)
\Bigr)_n \Bigr)=0.
\end{eqnarray*}
Setting 
%\begin{eqnarray}\label{ell}
$\displaystyle{l(A) := \bigvee_{m=1}^{\infty} 
\int_A h_m(g) \, q_m(g) \, d\mu(g)}$ for every $ A \in {\mathcal A}$,
%\end{eqnarray}
from the previous equality
%(\ref{conclusionproduct}) 
we deduce that 
the map $l$ %defined in (\ref{ell}) 
satisfies (\ref{l}).
%This ends the proof. 
\end{proof}
\begin{proof}[Proof of Corollary \ref{intproduct2}]
For each $g\in G$, set
%\begin{eqnarray*}
$h^*(g) = h(g) \chi_{\overline{B}}$.
It is not difficult to check that
$h^*$ is bounded and integrable on $G$. 
By Proposition \ref{intproduct}, the function 
$h^* \cdot q$ is bounded and integrable on $G$ too.
By hypothesis, we have
%\begin{eqnarray}\label{hq}
$h(g)\, q(g) = h^*(g) \, q(g) \chi_{\overline{B}}$.
%\end{eqnarray} 
Thus, the boundedness and integrability of $h \cdot q$ follow
from those of $h^* \cdot q$.
\end{proof}
%\begin{proof}

\begin{proof}[Proof of Proposition \ref{compuc}]
Let $u \in {\mathbf{X}}_1^+$ be related to the uniform continuity
of $f$ on $G$, and choose arbitrarily $
\varepsilon \in  \mathbb{R}^+$.
By the uniform continuity of $\psi$ on $\mathbf{X}_1$, there are $\sigma (\varepsilon)\in 
\mathbb{R}^+$
and $w \in {\mathbf{X}}_1^+$ with 
%\begin{eqnarray*}\label{compuc1}
$\vert \psi(x_1)-\psi(x_2)\vert  \leq_{X_1}  \, \varepsilon \, w$  
whenever $\vert x_1-x_2\vert \leq_{X_1} \sigma(\varepsilon) \, u$.
%\end{eqnarray*}
By virtue of the uniform continuity of $f$ on $G$, in correspondence with $\sigma(\varepsilon)$
there is $\delta(\varepsilon) \in \mathbb{R}^+$ with
%\begin{eqnarray*}\label{compuc2}
$\vert f(g_1)-f(g_2) \vert \leq_{X_1} \sigma (\varepsilon)\, u $ whenever  
$d(g_1,g_2) \leq \delta(\varepsilon)$.
%\end{eqnarray*}
So, for $x_i = f(g_i), \,i=1,2$, %From (\ref{compuc1}) and (\ref{compuc2})
it follows that 
%\begin{eqnarray*}
$\vert \psi(f(g_1))-\psi(f(g_2))\vert
 \leq_{X_1} \varepsilon \, w$ if $ d(g_1,g_2) \leq  \delta(\varepsilon).$
%\end{eqnarray*}
%getting the assertion. 
\end{proof}
\begin{proof}[Proof of Proposition \ref{boundedness}]
We begin with proving the boundedness of $f$ on $G$.
Let $C \subset G$ be a compact set
such that $f(g)=0$ whenever $g\in G \setminus C$. Observe that
it is enough to prove
that $f$ is bounded on $C$. As $f$ is uniformly continuous
on $G$, there exists $u \in {\mathbf{X}}^+_1 \setminus \{0\}$ 
such that for  $\varepsilon=1$ there is $\delta_1 \in \mathbb{R}^+$
with
\begin{eqnarray}\label{n0}
\bigvee_{g_1, g_2 \in 
G, \,\, d(g_1,g_2)\leq \delta_1} \vert f(g_1) - f(g_2) \vert  \leq_{X_1} u.
\end{eqnarray}
Since $C$ is compact, then $C$ is totally bounded too, and thus 
there is a finite number $t_1$, $t_2, \ldots, t_q$ of 
%\mg{\tiny e modifiche successive}
elements of $G$ such that 
$
\displaystyle{C \subset \bigcup_{j=1}^q 
%B(t^{(j)}, \sigma_{n_0})= \bigcup_{j=1}^q 
\{g \in G: d(g,t_j) \leq \delta_1\} }.
$
 Let 
$\displaystyle{\widehat{u}= \bigvee_{j=1}^q \,
\vert f(t_j) \vert}$,  and
pick arbitrarily $g \in G$. There is $\overline{j}
\in \{1,2,\ldots, q \}$ with
$d(g,t_{\overline{j}}) \leq \delta_1$. We have
\begin{eqnarray}\label{bdd}
\vert f(g) \vert  \leq_{X_1} \vert f(g) - f(t_{\overline{j}}) \vert + 
\vert f(t_{\overline{j}}) \vert \leq_{X_1} u+ \widehat{u}.
\end{eqnarray} 
Thus, $f$ is bounded on $C$.
%==================== vecchia proof======================
%==================================================

%======================================
%=======================================
Now we prove the integrability of $f$ on $G$.
Without loss of generality, we can assume $f \geq 0$.
Let $(\epsilon_n)_n$ be any $(o)$-sequence 
in $\mathbb{R}^+$. 
By the uniform continuity of $f$ on $G$, 
there exists an $(o)$-sequence $(\delta_n)_n$ in
$\mathbb{R}^+$ with 
\begin{eqnarray}\label{uc1}
0 \leq_{X_1}\vert f(g_1) - f(g_2)\vert  \leq_{X_1} \varepsilon_n \, u \quad \text{whenever  } d(g_1,g_2) 
\leq \delta_n
\end{eqnarray}
for every $n \in \mathbb{N}$.
% where $u \in {\mathbf{X}}^+_1 \setminus \{0\}$ is as in (\ref{n0}). 
From  \ref{ass}.7)
 we deduce
\begin{eqnarray}\label{sussmayrconditionucgenu}
\displaystyle{\ell_1 ((\varepsilon_n \, u  )_n)=0}. 
\end{eqnarray} 
Now we construct a  defining
sequence  by induction. 
At the generic $n$-th step, $n \in \mathbb{N}$, we partition the set $C$ 
into a finite number $l_n$ of pairwise disjoint sets of ${\mathcal A}$ 
of diameter less than or equal to 
$\delta_n$ (this is always possible, since
$G$ is a metric space and $C$ is totally bounded) with $\mu(A) \in X_2$ for every $A \in \mathcal{E}_n$, thanks to the regularity of $\mu$. 

Denote this family by  
${\mathcal E}_n=\{E^{(j)}_n: j=
1,2, \ldots l_n \}$, and take ${\mathcal E}_n$ in such a way that 
${\mathcal E}_n$ is a refinement of the family ${\mathcal E}_{n-1}$ constructed 
at the $n$-$1$-th step: that is, for each 
$E^{(j)}_{n} \in {\mathcal E}_n$, $j=1,2,\ldots, l_n$ there exists 
$E^{(i)}_{n-1} \in {\mathcal E}_{n-1}$ with $E^{(j)}_n \subset  E^{(i)}_{n-1}
$. For every $n \in \mathbb{N}$, set 
\begin{eqnarray}\label{mozartsimpleuniform1}
f_n(g)= \bigwedge \{f(g):g \in E^{(j)}_n\} \quad \text{ if   } g \in E^{(j)}_n
\quad (j=1,2, \ldots, l_n).
\end{eqnarray}
Moreover, let $f_n(g)=0$ for each $n \in \mathbb{N}$ and 
$g \in G \setminus C$.
It is not difficult to see that, for every $ g \in G$
  and  $n \in \mathbb{N}$,
\begin{eqnarray}\label{pp}
f_n \in\mathscr{S}\, \quad  \text{   and   }\quad   0 
 \leq_{X_1} f_n(g) \leq_{X_1}
f_{n+1}(g)  \leq_{X_1} f(g)  \leq_{X_1} u+ \widehat{u}.
\end{eqnarray} 
From (\ref{uc1}) and (\ref{mozartsimpleuniform1}) we obtain
\begin{eqnarray*}\label{mozartuniform}
0  \leq_{X_1} \bigvee_{g \in G} (f(g)-f_n(g)) \leq_{X_1} 
\bigvee_{g_1, g_2 \in G, d(g_1,g_2)\leq 
\delta_n} \vert   f(g_1) - f(g_2) \vert 
 \leq_{X_1} \varepsilon_n \, u,
\end{eqnarray*}
and then by 
%From (\ref{sussmayrconditionucgenu}), (\ref{mozartuniform})%and the 
Axioms \ref{convergenze} we get
\begin{eqnarray*}
\displaystyle{\ell_1\Bigl(\Bigl( \bigvee_{g \in G} \,
\vert f_n(g)-f(g) \vert \Bigr)_n\Bigr)=0}.
\end{eqnarray*}
Thus, the sequence $(f_n)_n$ converges uniformly 
to $f$ on $G$,
and a fortiori it converges in measure to $f$ on $G$.

Now we prove the $\mu$-equiabsolute continuity of 
the integrals of the $f_n$'s. We begin with \ref{eac}.1).
%Let $u^*=u+\widehat{u}$, where $u$ and  $\widehat{u}$ are as in (\ref{bdd}).
Choose arbitrarily a sequence $(A_n)_n$ from ${\mathcal A}$, with
$\ell_2((\mu(A_n))_n)=0$.
%From (\ref{bdd}) and
By  (\ref{pp}) we get 
\begin{eqnarray*}\label{integralequiboundedness}
0  \leq_{X} \int_{A_n} f_n(g) \, d\mu(g) \leq_{X} 
\mu(A_n) ( u+\widehat{u})
\quad \text{ for  any   } n \in \mathbb{N}. 
\end{eqnarray*}
From the previous inequality,   
Axioms \ref{convergenze}.b) and \ref{convergenze}.c)
and condition \ref{ass}.7) it follows that 
\begin{eqnarray*}\label{eaceac1}
\displaystyle{
\ell \left( \left(\int_{A_n}\,f_n(g) \,
d\mu(g)\right)_n \right)=0}.
\end{eqnarray*}
Now we turn to \ref{eac}.2). From (\ref{pp}), the positivity
of $\mu$ and the monotonicity of the integral we deduce 
\begin{eqnarray}\label{sussmayrcompactsupport}
0  \leq_{X} \int_{G \setminus C} f_n(g) \, d\mu(g)
 \leq_{X} \int_{G \setminus C} f(g) \, d\mu(g)=0  
\quad \text{ for   all  }  n \in \mathbb{N}.
\end{eqnarray} 
Condition \ref{eac}.2) follows
from (\ref{sussmayrcompactsupport}) and 
 Axioms \ref{convergenze}, taking $B_m=C$ 
for each $m \in \mathbb{N}$.

Now we prove the existence of a map $l$, satisfying (\ref{l}) in the definition of integrability. 
%From (\ref{pp}) and (\ref{propertyproblemsolving}), 
 For each 
%$m \geq \overline{m}$, 
$n \geq m$
and $A \in {\mathcal A}$, we get
\begin{eqnarray*}
0 & \leq_{X}& \int_A f_n(g) \, d\mu(g) -\int_A f_m(g) \, d\mu(g) =\nonumber
\int_A (f_n(g) -f_m(g)) \, d\mu(g)  \leq_{X}  \\ & \leq_{X}&
\Bigl( \bigvee_{g \in C} (f_n(g) - f_m(g))  \Bigr) \, 
\mu(C)  \leq_{X}
\Bigl( \bigvee_{g \in C} (f(g) - f_m(g))  \Bigr) \, 
\mu(C)  \leq_{X}  \nonumber \,
 \varepsilon_n \,\mu(C) \, u.
\end{eqnarray*}
Taking in the previous formula
the supremum as $n$ varies in $\mathbb{N}$, 
and taking into account the
monotonicity of the sequence $(f_n)_n$
  and of the integral, we obtain
\begin{eqnarray*}
0 &  \leq_{X}& \bigvee_{A \in {\mathcal A}}
\Bigl(\Bigl(   \bigvee_{n=1}^{\infty} 
\int_A f_n(g) \, d\mu(g) \Bigr) -\int_A f_m(g) \, d\mu(g) \Bigr)
 \leq_{X}\varepsilon_n \, \mu(C) \, u .
\end{eqnarray*}
Put 
%\begin{eqnarray}\label{lel}
$l(A)= \bigvee_{n=1}^{\infty} 
\int_A f_n(g) \, d\mu(g)$, $A \in {\mathcal A}$.
%\end{eqnarray}
From (\ref{sussmayrconditionucgenu}), the previous inequality 
%(\ref{lel})
 and Axioms 
\ref{convergenze} we obtain
\begin{eqnarray*}
\ell \Bigl( \Bigl( \bigvee_{A \in \mathcal{A}}
\Bigl( l(A) - \int_A \, f_m(g)\,d\mu(g) \Bigr)
\Bigr)_m \Bigr)=0.
\end{eqnarray*}
%that is (\ref{l}). This ends the proof. $\quad \Box$
\end{proof} 
\begin{proof}[Proof of Proposition \ref{beginning}]
By 
%(\ref{support}) used 
the definition of convexity, if we consider 
%with
 $v= \xi x$, $s=0$, 
there is $\beta_{\xi x} \in \mathbf{X}_1$ with
%\begin{eqnarray*}\label{mozart01}
$0 = \varphi(0) \geq_{X_1} \varphi(\xi x) - \beta_{\xi x} \, \xi x,$ 
%\end{eqnarray*} 
while for 
%By applying again (\ref{support}) with 
$v= \xi x$ and $s=x$, we get
%\begin{eqnarray*}\label{mozart02}
$\varphi(x) \geq_{X_1}  \varphi(\xi x) + \beta_{\xi x} (1-\xi) x$.
%\end{eqnarray*}
Multiplying by $1-\xi$ the first one 
%all members of (\ref{mozart01}) 
and by $\xi$ the second one, 
%all members of (\ref{mozart02}), 
we obtain 
\begin{eqnarray*}\label{mozart03}
0 \geq_{X_1}  (1-\xi)  \varphi(\xi x) - \beta_{\xi x} \, \xi (1-\xi) x ,
\quad
\xi \, \varphi(x) \geq_{X_1}  \xi \, \varphi(\xi x) + \beta_{\xi x} \, \xi (1-\xi) x .
\end{eqnarray*}
Summing up,  
%(\ref{mozart03}) and (\ref{mozart04}), 
we get\,
$\xi \, \varphi(x) \geq_{X_1}  \varphi(\xi x)$.%, that is (\ref{mozart0}).
\end{proof}

\begin{proof}[Proof of Proposition \ref{locallipschitzconvex}]
Fix arbitrarily $u \in \mathbf{X}_1$, 
%\mg{\tiny messo $X_1 \, \rightarrow$}
pick $v \in [-u,u]$, let $\beta_v$ be as in Definition \ref{def-conv}
and set $\displaystyle{\beta^*_u =\bigvee_{v \in [-u,u]}
\vert \beta_v \vert}$. 
%By (\ref{betabeta}), 
 Note that
$\beta^*_u \in \mathbf{X}_1$.
Choose arbitrarily $x_1$, $x_2 \in \mathbf{X}_1$.
By the convexity of $\varphi$, if we set
% (\ref{support}) used with
 $v=x_1$ and $s=x_2$
(resp., $v=x_2$ and $s=x_1$), we get
\begin{eqnarray*}\label{support12}
\varphi(x_2) \geq_{X_1}  \varphi(x_1) 
+ \beta_{x_1} (x_2-x_1),%\\ 
\quad \text{(resp., }
\varphi(x_1) \geq_{X_1}  \varphi(x_2) + \beta_{x_2} (x_1-x_2) \, ).
%\nonumber
\end{eqnarray*}
So,
$\vert \varphi(x_2) - \varphi(x_1)\vert  \leq_{X_1}   \beta^*_u \, \vert x_1-x_2\vert $.
% that is the assertion. 
\end{proof}
%==================================================

\begin{proof}[Proof of Proposition \ref{compositionmozarteum}]
Thanks to the uniform continuity of $f$ on $G$, there is an element 
%\mg{\tiny messo $X_1\, \rightarrow$}
$u \in  \mathbf{X}_1^+ \setminus \{0\}$ with the property that for every
$\varepsilon \in \mathbb{R}^+ $ there exists $
\delta \in \mathbb{R}^+$
with 
\begin{eqnarray*}\label{ucuc}
\vert f(g_1) - f(g_2)\vert  \leq_{X_1}  \varepsilon \, u \quad \text{ whenever  }
g_1, g_2 \in G,   \, \, d(g_1,g_2)\leq \delta.
\end{eqnarray*}
Moreover, since $f(g)=0$ for all $g \in G \setminus C$  and $C$ is compact, by virtue of  Proposition \ref{boundedness} we get that
$f$ is bounded. Let 
$\displaystyle{u^*=\vee_{g \in G} \vert f(g) \vert }$.
By Proposition \ref{locallipschitzconvex}, there is an element $\beta^*_{u^*} \in \mathbf{X}_1^+ \setminus \{0\}$ with 
%\begin{eqnarray*}\label{convexlipschitz2}
$\vert \varphi(x_1)-\varphi(x_2) \vert \leq_{X_1}  \beta^*_{u^*} \, \vert x_1 - x_2\vert $ for all $x_1, x_2 \in [-u^*,u^*]$. 
Then, it follows that
\begin{eqnarray*}\label{uccomp}
\vert \varphi(f(g_1)) - \varphi(f(g_2))\vert  \leq_{X_1}   \beta^*_{u^*} \, \vert (f(g_1)-f(g_2))\vert \leq_{X_1} 
\varepsilon \, \beta^*_{u^*} \, \, u
\end{eqnarray*}
whenever $g_1$, $g_2 \in G$, $d(g_1,g_2) \leq \delta$.
Therefore, $\varphi \circ f$ is uniformly continuous  on $G$. Since $f(g)=0$ for every $g \in G \setminus C$ and $\varphi(0)=0$, we get that
$\varphi \circ f$ vanishes outside $C$. Thus, $\varphi \circ f$ is integrable on $C$, thanks to Proposition \ref{boundedness}.
\end{proof}
\begin{proof}[Proof of Theorem \ref{jensen1}]
By Proposition \ref{lebesgue} (resp., \ref{boundedness}), $h$ 
(resp., $f$) is integrable on $G$. 
%, according to Definitions \ref{integrabilita} and \ref{integrabilitaa}. 
Moreover, taking into account that $\varphi(0)=0$ and since  
$({\mathbf{X}},\mathbb{R}, {\mathbf{X}})$ is a product triple,
thanks to Proposition \ref{compositionmozarteum}, we have that $\varphi \circ f$ is  uniformly continuous and integrable on $G$  and vanishes outside $C$.
If we apply twice 
 Corollary \ref{intproduct2}  to the pairs $(h,f)$ and  $(h,\varphi \circ f)$ 
it follows that the functions $h \cdot f$  and  $h \cdot (\varphi \circ f)$ are bounded and integrable on $G$.
Let 
$\tau= \int_G h(g) \, f(g) \, d\mu(g) \in \mathbf{X}.$
%(resp. $\displaystyle{\zeta= \int_G h(g) \, \varphi(f(g)) \, d\mu(g)} \, )$
%is well-defined.
Since $\varphi$ is convex,  in correspondence with $\tau$ there is an element $\beta_{\tau}  \in \mathbf{X}$,
satisfying the definition  of convexity
% (\ref{support}), 
with $s=f(g)$, $g \in G$, %that is 
\begin{eqnarray*}\label{72bis}
	\varphi(f(g)) \geq_{X}  \varphi(\tau) +  \beta_{\tau} \,  (f(g)-\tau).
\end{eqnarray*}
By multiplying both members of %(\ref{72bis}) 
by $h(g)$
%, since $h(g) \geq 0$
 we get,  for all $g \in G$,
\begin{eqnarray*}\label{sussmayr} 
	h(g) \,
	\varphi(f(g)) \geq_{X} h(g) \, \varphi(\tau)  + h(g) \, \beta_{\tau} \, f(g) -\, h(g) \, \beta_{\tau} \, \tau.
\end{eqnarray*}
Observe that all members of the previous inequality are integrable on $G$.
Since the integral is monotone,
 (see Remark
%\ref{pf}.e), 
\ref{pf}.f) and
$\displaystyle{\int_G h(g) \, d\mu(g)=1}$, we get
\begin{eqnarray*}
	&& \int_G h(g) \, \varphi(f(g)) \,  d\mu(g) \geq_{X} \\
	&&\geq_{X} \Bigl( \int_G h(g) \, d\mu(g) \Bigr) \, 
	\cdot \, \Bigl( \varphi \Bigl(\int_G h(g) \, f(g) \, d\mu(g) \Bigr) \Bigr)+
	\beta_{\tau} \int_G h(g) \, f(g) d\mu(g) -\\ &&- \beta_{\tau} \Bigl( \int_G h(g) \, d\mu(g) \Bigr) \, \cdot \,
	\Bigl(\int_G h(g) \, f(g) \, d\mu(g) \Bigr) = \varphi \Bigl (\int_G h(g) \, f(g) \, d\mu(g) \Bigr).
\end {eqnarray*}

Therefore, the assertion follows.
%from the monotonicity of $\ell$ (Axiom \ref{convergenze}.b).
%that is (\ref{jensenmozart}).
\end{proof}
\begin{proof}[Proof of Corollary \ref{jensencoroll}]
Let $\displaystyle{m_h=\int_G h(g) \, d\mu(g)}$, and  for every $g \in G$, set $h^*(g)=h(g)/m_h$.
%From (\ref{schumann0}) it follows that  
By hypothesis, $\displaystyle{\int_G h^*(g) \, d\mu(g)=1}$.
From Theorem \ref{jensen1} applied to $h^*$, we obtain
\begin{eqnarray*}\label{sussmayr1} \nonumber
	\dfrac{1}{m_h}\int_G h(g) \, \varphi(f(g)) \, d\mu(g) &=&
	\int_G h^*(g) \, \varphi(f(g)) \,  d\mu(g) \geq_{X}
\\
	&\geq_{X}&
 \varphi \Bigl (\int_G h^*(g) \, f(g) \, d\mu(g) \Bigr) \geq_{X}  \varphi \Bigl(\dfrac{1}{m_h} \int_G h(g) \, f(g) \, d\mu(g) \Bigr).
\end{eqnarray*}
From 
Proposition \ref{beginning} applied to $\varphi$,
%(\ref{mozart0})
  with $\displaystyle{x= \dfrac{1}{m_h}\int_Gh(g) \, f(g) \, d\mu(g)}$ and $\xi=m_h$, we get
\begin{eqnarray*}\label{sussmayr2}
	\varphi \Bigl(\dfrac{1}{m_h} \int_G h(g) \, f(g) \, d\mu(g) \Bigr)
	\geq_{X} \dfrac{1}{m_h} \, \varphi \Bigl (\int_G h(g) \, f(g) \, d\mu(g) \Bigr).
\end{eqnarray*}
%Thus, from (\ref{sussmayr1}) and (\ref{sussmayr2}) we get
 Therefore, we obtain
\[\int_G h(g) \varphi(f(g))  \, d\mu(g) \geq_{X}
\varphi \Bigl (\int_G h(g) f(g) \, d\mu(g) \Bigr).\]
%that is the assertion.
\end{proof}

\end{document}